\documentclass[a4paper,10pt]{amsart}

\usepackage[utf8]{inputenc}
\usepackage[english]{babel}

\usepackage[top=35mm,bottom=35mm,left=35mm,right=35mm]{geometry}

\numberwithin{equation}{section}
\usepackage{amsmath,amsthm,amsfonts,amssymb}
\usepackage{euscript}
\usepackage{mathtools}
\mathtoolsset{showonlyrefs}

\usepackage[
   pdfstartview=FitH,
   pdfpagemode=UseNone,
]{hyperref}
\usepackage[textsize=footnotesize,color=orange!30]{todonotes}
\theoremstyle{plain}
\newtheorem{theorem}{Theorem}[section]
\newtheorem{thm}[theorem]{Theorem}
\newtheorem{lemma}[theorem]{Lemma}
\newtheorem{lem}[theorem]{Lemma}
\newtheorem{corollary}[theorem]{Corollary}
\newtheorem{proposition}[theorem]{Proposition}

\newtheorem*{rem*}{Remark}

\theoremstyle{definition}

\newtheorem{remark}[theorem]{Remark}

\def\MR#1{\href{http://www.ams.org/mathscinet-getitem?mr=#1}{MR#1}}

\def\RR{\mathbb{R}}

\def\CC{\mathbb{C}}

\def\Ocal{\mathcal{O}}

\def\eps{\epsilon}

\def\sym{\mathrm{sym}}
\def\Cl{\mathrm{Cl}}

\newcommand{\tmop}[1]{\ensuremath{\operatorname{#1}}}

\def \ck {{\xi}} 
\def \e {{\varepsilon}}
\def \a {{\mathfrak{a}}}

\def \R {{\mathbb R}}
\def \H {{\mathbb H}}

\def \Q {{\mathbb Q}}

\def \C {{\mathbb C}}
\def \g {\gamma}
\def \G {\Gamma}
\def \Z {\mathbb{Z}}

\def \slz  {{\mathrm{SL}_2( {\mathbb Z})} }
\def \GmodH {{\Gamma\backslash\mathbb H}}

\def \supp {{\rm supp\,} }

\DeclareMathOperator{\arccosh}{arcosh}

\newcommand{\vol}[1]{\mathrm{vol}\left( #1 \right)}
\newcommand{\abs}[1]{\left\lvert #1 \right\rvert}
\newcommand{\norm}[1]{\left\lVert #1 \right\rVert}
\newcommand{\Mod}[1]{\ (\textup{mod}\ #1)}

\providecommand{\sym}{\operatorname{sym}}

\title{The hyperbolic circle problem over Heegner points}

\author[Chatzakos]{Dimitrios Chatzakos}
\address{
   Department of Mathematics,
   University of Patras,
   26 504, Patras,
   Greece
}
\email{dchatzakos@math.upatras.gr}

\author[Cherubini]{Giacomo Cherubini}
\address{
   Istituto Nazionale di Alta Matematica ``Francesco Severi'',
   Research Unit Department of Mathematics ``Guido Castelnuovo'',
   Sapienza University of Rome,
   Piazzale Aldo Moro 5, I-00185, Rome,
   Italy
}
\email{cherubini@altamatematica.it}

\author[Lester]{Stephen Lester}
\address{
   Department of Mathematics \\
   King's College London \\
   London WC2R 2LS \\
   UK
}
\email{steve.lester@kcl.ac.uk}

\author[Risager]{Morten S. Risager}
\address{
   Department of Mathematical Sciences,
   University of Copenhagen,
   Universitets\-parken 5, 2100
   Copenhagen \O, Denmark
}
\email{risager@math.ku.dk}

\date{\today}


\begin{document}

\begin{abstract}
  For the full modular group, we obtain a logarithmic improvement  on Selberg's long-standing bound
  for the error term of the counting function
  in the hyperbolic circle problem
  over Heegner points of different discriminants.
  The main ingredients in our method are Waldspurger's formula,
  twisted first moments of certain Rankin-Selberg convolutions,
  and a new fractional moment estimate.
\end{abstract}

\maketitle

\section{Introduction}

\subsection{Background}
In this paper we study the hyperbolic
lattice point counting problem
for the modular group $\Gamma = \slz$.
Let $\mathbb{H}$ be the hyperbolic plane. For two fixed points $z$, $w$ in $\mathbb{H}$ we denote by $d_\H(z,w)$ their hyperbolic distance.
The hyperbolic
circle problem
asks to estimate the asymptotic growth of the counting function
\begin{equation}
  N(X;z, w)  = \# \{ \gamma \in \Gamma : 4u( z, \gamma w) + 2 \leq X \}
\end{equation}
as $X \to \infty$, where the point-pair invariant function
\begin{equation}
  u(z,w) = \frac{|z-w|^2}{4 \Im(z) \Im(w)}
\end{equation}
is related to the hyperbolic distance by the relation $\cosh d_\H(z,w) = 2 u(z,w) +1$.
We refer to \cite{PhillipsRudnick:1994, Chamizo:1996a, Iwaniec:2002}
for the very rich history of this problem.
The main theorem, proved first by Selberg \cite{Selberg:1977},
states that
\begin{equation}\label{selbergbound}
  N(X;z, w)  = \frac{2 \pi}{\hbox{vol}(\GmodH)} X + O(X^{2/3}),
\end{equation}
where the implied constant may depend on $z,w,$ and $\Gamma$.
The same estimate holds for any cofinite Fuchsian group $\Gamma$, if one adjusts the main term when there are small eigenvalues
(see \cite[Theorem 12.1]{Iwaniec:2002}, \cite[Theorem 4]{Good:1983});  similar estimates were recently found for higher rank groups by Blomer and Lutsko \cite{BlomerLutsko:2024}.
The error term $O (X^{2/3})$ has never been improved
for any cofinite group $\Gamma$ or for any pair of points $z, w$,
although the main conjecture predicts that for any $\varepsilon>0$
\begin{equation} \label{mainconjecture}
  N(X;z, w)  = \frac{2 \pi}{\hbox{vol}(\GmodH)} X + O(X^{1/2+\varepsilon}).
\end{equation}
This conjecture is supported by second moment estimates of Chamizo \cite{Chamizo:1996a}, and by the work of Phillips and Rudnick \cite{PhillipsRudnick:1994}, who proved mean value and $\Omega$-results for the error term. We also refer to \cite{Chatzakos:2017, Cherubini:2018, CherubiniRisager:2018, PetridisRisager:2018a} for further results supporting the validity of the conjecture \eqref{mainconjecture}.

\subsection{Main results and strategy of proof}
We focus on the case of the modular group. Our main result gives the first unconditional improvement of Selberg's bound \eqref{selbergbound} in the case where $z,w$ are Heegner points of different discriminants (see Section \ref{sec:heegner}).
\begin{thm} \label{thm:mainresult}
  Let $\G=\slz$ be the modular group and $z_d,z_{d'}\in \H$ be  Heegner points of different squarefree discriminants $d,d'<0$ respectively. Then
  \begin{equation}
    N(X;z_d, z_{d'})  = \frac{ 2 \pi}{\vol{\GmodH}} X + O\left(\frac{X^{2/3}}{(\log X)^{1/6} }\right).
  \end{equation}
\end{thm}
In order to explain the strategy of the proof, we recall some notation. Let $-\Delta$ be the (positive) automorphic Laplacian of the modular surface $\GmodH$
and let $\{ \phi_j \}_{j=0}^{\infty}$
denote an orthonormal basis for the cuspidal spectrum of $-\Delta$ with eigenvalues $ \{\lambda_j \}_{j=0}^{\infty}$.
As usual,
we write the non-zero eigenvalues as $\lambda_j = 1/4 +t_j^2$ with $t_j > 0$.
Since we restrict ourselves to the modular group there are no small eigenvalues $\lambda_j \in (0,1/4]$.

Bounding the error term in the hyperbolic circle problem is closely related to bounding the spectral exponential sum given by
\begin{equation}
  \mathfrak{S} (T,X) := \sum_{t_j \leq T} X^{it_j} \phi_j(z) \overline{\phi_j(w)}.
\end{equation}
Using Cauchy--Schwarz and the local Weyl law
(see \cite{Hormander:1968}, \cite[Theorem 7.2]{Iwaniec:2002}) we get the \lq trivial\rq{}  bound $\mathfrak{S} (T,X)\ll T^2$. The same bound holds if we replace the modular group by any
cofinite
$\G$ and any $z,w\in \H$. This trivial bound  suffices to prove Selberg's bound \eqref{selbergbound} in this generality. The difficulty in improving \eqref{selbergbound} largely arises from the inability to improve on the trivial bound on $\mathfrak{S} (T,X)$ in the relevant ranges.

In contrast, when $z\neq w$ are different points and $X=1$, a well-known bound of H\"ormander's  (see \cite{Hormander:1968} for cocompact groups and \cite{Chamizo:1996} for cofinite groups) gives that in this case the spectral exponential sum is significantly smaller
and we have
\begin{equation} \label{spectralfunction}
  \sum_{t_j \leq T}  \phi_j(z) \overline{\phi_j(w)} \ll T^{1+\varepsilon}.
\end{equation}
In fact, H\"ormander's bound holds for Laplace eigenfunctions over any compact Riemannian manifold. The study of the correct order of growth of the spectral function \eqref{spectralfunction} has attracted a significant amount of interest among geometric analysts.
It is possible that the true order may be as small as $T^{1/2+\varepsilon}$, at least under some restrictions on $z$, $w$, see \cite{CanzaniHanin:2015, JakobsonPolterovich:2007, LapointePolterovichSafarov:2009}.
A key step  in proving Theorem \ref{thm:mainresult} gives a bound for a related spectral sum,
namely
establishing the bound
\begin{equation}\label{eq:crucial-bound}
  \sum_{t_j \leq T}  \abs{\phi_j(z)\phi_j(w)}
  \ll \frac{T^2}{(\log T)^{1/4}}
\end{equation}
for the full modular group when $z =z_d, w=z_{d'}$ are Heegner points of \emph{different} discriminants and $\{\phi_j\}_{j=0}^{\infty}$ consists of Hecke--Maass cusp forms. This bound exploits the underlying arithmetic at the Heegner points and is stronger than what Berry's random wave model suggests should hold generically if one assumes that $\phi_j(z)$ and $\phi_j(w)$ are independent for $t_j\leq T$.
A similar phenomenon occurs
in the sup-norm problem, where Berry's model predicts that the sup-norm of $\phi_j$ should be $\ll \sqrt{\log \lambda_j}$ whereas Mili{\'c}evi{\'c} \cite[Theorem 1]{Milicevic:2010} proved for each Heegner point $z_d$ there exists a subsequence of Laplace eigenfunctions with $|\phi_j(z_d)|
  \gg \exp(\sqrt{ \frac{\log \lambda_j}{\log \log \lambda_j}}(1-o(1)))$.

In order to prove \eqref{eq:crucial-bound} we use Waldspurger's formula to express $\abs{\phi_j(z)}$ through $L$-functions of
Rankin--Selberg convolutions.
This gives expressions which may be approached using techniques from analytic number theory. What we then prove is a new fractional moment estimate for twisted $L$-functions of Hecke--Maass forms which may be of independent interest. To state it we let $f_\ck$ be the  theta series associated to a class group character, $\ck\in \widehat{\Cl}_{K}$, of an imaginary quadratic field $K$ (see Section \ref{sec:theta-series}).

\begin{thm} \label{thm:fracmoments}
  Let $d,d'\equiv 1 \bmod 4$ be distinct negative, squarefree integers.
  Let $\ck \in \widehat{\Cl}_{\mathbb Q(\sqrt{d})}$
  and $\ck' \in \widehat{\Cl}_{\mathbb Q(\sqrt{d'})}$. Then
  \begin{equation} \label{eq:fractional-moment-bd}
    \sum_{T < t_j \le 2T} \frac{L(\tfrac12,\phi_j \times f_{\ck})^{\tfrac 1 2} L(\tfrac12,\phi_j \times f_{\ck'})^{\tfrac 1 2}}{L(1,\tmop{sym}^2 \phi_j)}  \ll \frac{T^2 }{(\log T)^{1/4}},
  \end{equation}
  where the implied constant depends at most on $d$ and $d'$ and $\{\phi_j\}_{j=0}^{\infty}$ consists of Hecke-Maass cusp forms.
\end{thm}

Some comments are in order:
\begin{enumerate}
  \item The study of fractional moments originates in works of Ramachandra and  Heath-Brown, in the context of the Riemann zeta-function \cite{Ramachandra:1980, Ramachandra:1980a,Heath-Brown:1981} and Dirichlet $L$-functions \cite{Heath-Brown:2010a}. In recent years, further fractional moment estimates of $L$-functions have emerged as a powerful new arithmetic tool, having their origins in the breakthrough works of Soundararajan \cite{Soundararajan:2009}, Harper \cite{Harper:2013}, and Radziwi{\l}{\l} and Soundararajan \cite{RadziwillSoundararajan:2015a}.
        There have been a number of subsequent works, which build upon these methods and we refer the reader to the survey article of Soundararajan \cite{Soundararajan:2023}. Estimates for fractional moments have been applied to various equidistribution problems, such as the quantum unique ergodicity conjecture (by Lester and Radziwi{\l}{\l} \cite{LesterRadziwill:2020}), equidistribution of toric periods (by Blomer and Brumley \cite{BlomerBrumley:2024}), and statistical independence of Hecke--Maass cusp forms (by Hua, Huang and Li \cite{HuaHuangLi:2024}).
        All these results use that fractional moments of the $L$-functions over the families considered are small, as predicted by the heuristics of Keating and Snaith \cite{KeatingSnaith:2000}. In our setting we also use that for non-genus characters $\ck \neq \ck'$ the two central $L$-values should interact independently from one another so that the fractional moment of the product should also be small (this feature also appears in the genus case however the $L$-function will factor into either three or four distinct $L$-functions provided $d \neq d'$ all of which should interact independently from one another).

  \item
        If $\ck$ is a genus character the $L$-function $L(s,\phi_j \times f_{\ck})$ factors into a product of two different $L$-functions (see Section \ref{sec:real-characters} and \eqref{eq:RS-factorization-genus-characters}) and when at least one of $\ck,\ck'$ is a genus character we improve upon the bound \eqref{eq:fractional-moment-bd} in certain cases.  When precisely one of $\ck,\ck'$ is a genus character we prove the stronger bound $\ll T^2(\log T)^{-3/8}$. If $\ck',\ck$ are both genus characters
        corresponding to $d=d_1d_2$ and $d'=d_1'd_2'$ with $d_1,d_2,d_1',d_2'$ distinct integers, 
        $L(s,\phi_j \times f_{\ck}) L(s,\phi_j \times f_{\ck'})$ factors into a product of four different $L$-functions, and we show that the fractional moment is $\ll T^2(\log T)^{-1/2}$ in this case. This does not, however, allow us to improve on Theorem \ref{thm:mainresult} as the trivial characters will always contribute with the bound stated in Theorem \ref{thm:fracmoments}.
        Following the heuristics of Keating and Snaith \cite[Section 3.2]{KeatingSnaith:2000} and assuming the different central $L$-values interact independently of one another over $T<t_j \le 2T$, these upper bounds are expected to be sharp.
\end{enumerate}

Our proof of Theorem \ref{thm:fracmoments} relies on a new twisted first moment estimates which provides asymptotics and good error terms for  sums like
\begin{equation}\label{eq: twisted first moment}
  \sum_{t_j>0} \frac{L(\tfrac12,\phi_j\times f_\ck)\lambda_{j}(\ell)}{L(1,\sym^2\phi_j)}h(t_j)
\end{equation}
where $h$ is a certain smooth function essentially supported around $T$ and $\lambda_j(\ell)$ is the $\ell$th Hecke eigenvalue of $\phi_j$. See Theorem \ref{thm:Twisted-first-moment} for the precise statement. Such twisted first moments have been studied by Hoffstein, Lee, and Nastasescu \cite{HoffsteinLeeNastasescu:2021}, but they have the simplifying assumption that the level of $f_\ck$ should divide the level of $\phi_j$ which is not satisfied in our case.  The analogous sums over the holomorphic family of weight $k$, level $N$ forms were considered by Michel and Ramakrishnan \cite{MichelRamakrishnan:2012}.
When $\ck$ is the trivial character and $\ell=1$, asymptotics for \eqref{eq: twisted first moment}
also appear in work by Humphries and Radziwi\l\l{} \cite{HumphriesRadziwill:2022}.

Equipped with asymptotics for \eqref{eq: twisted first moment}, in Section \ref{sectionfractional} we estimate the fractional moment using mollifiers, which behave like Euler products.
The following elementary inequality, valid for real numbers $M,M'>0$, $L, L' \ge 0$ is used
\begin{equation} \label{eq:elementary}
  2\sqrt{L L'} \le L M (M')^{-1}+L'M' M^{-1},
\end{equation}
which is related to the starting point for the key inequality in Radzwi{\l}{\l} and Soundararajan \cite[Section 3.1]{RadziwillSoundararajan:2017}.
We apply \eqref{eq:elementary} with $L=L(\tfrac12,\phi_j \times f_{\ck})$ and $L'=L(\tfrac12,\phi_j\times f_{\ck'})$, which is valid since the central $L$-values are non-negative (see \eqref{eq:Waldspurger}).
Noting the Euler product representation
\begin{equation}
  L(s,\phi_j\times f_{\ck})=\prod_p L_p(s,\phi_j \times f_{\ck}),
\end{equation}
which is explicitly given in \eqref{eq:eulerproduct}, a natural choice would be to take
\begin{equation}
  M=\prod_{p\le x} L_p(\tfrac12,\phi_j \times f_{\ck})^{-1/2} \quad \textrm{ and } \quad  M'=\prod_{p\le x} L_p(\tfrac12, \phi_j\times  f_{\ck'})^{-1/2}
\end{equation}
since both $M,M'$ and their inverses can be expressed as Dirichlet series.
To simplify the analysis, we choose $M,M'$ slightly differently. When $\ck^2,(\ck')^2 \neq 1$ we consider instead \begin{equation}M=(\log x)^{1/4} \exp(P(j,\ck))\quad \textrm{ and } \quad M'=(\log x)^{1/4} \exp(P(j,\ck'))\end{equation} where $P(j,\ck)=-\tfrac12\sum_{p\le x} a_{\phi_j \times f_{\ck}}(p)p^{-1/2}$ and $a_{\phi_j \times f_{\ck}}(n) $ is the $n$th coefficient of the Dirichlet series associated to $\log L(s,\phi_j \times f_{\ck})$. The factor $(\log x)^{1/4}$ accounts for the contribution from $\sum_{p \le x}a_{\phi_j \times f_{\ck}}(p^2)/p$, (see the discussion following \eqref{eq:molldef}).
Expanding out $M(M')^{-1},M'M^{-1}$ leads to
long Dirichlet polynomials unless $x$ is rather small. If $x=\log T$ then both $M(M')^{-1},M'M^{-1}$ are well-approximated by Dirichlet polynomials of length $\le T^{\varepsilon}$ for all $t_j \le T^{1+\varepsilon}$.
With this choice of $M,M'$ and $x=\log T$, we can bound the fractional moment in terms of mollified first moments. The estimates established in Section \ref{sectionfractional} for these mollified first moments lead to the bound
\begin{equation} \label{eq:weakbd}
  \sum_{T<t_j\le 2T}   \frac{L(\tfrac12,\phi_j \times f_{\ck})^{\tfrac 1 2} L(\tfrac12,\phi_j \times f_{\ck'})^{\tfrac 1 2}}{L(1,\tmop{sym}^2 \phi_j)}  \ll \frac{T^2}{(\log x)^{1/4}}=\frac{T^2}{(\log \log T)^{1/4}}.
\end{equation}
This provides a non-trivial estimate for the fractional moment and leads to an improvement on Selberg's bound \eqref{selbergbound} for $z,w$ satisfying the hypotheses of Theorem \ref{thm:mainresult}.

To improve upon \eqref{eq:weakbd} we would like to take $x$ larger and we will illustrate how to achieve this for $x=T^{1/(\log \log T)^4}$, say, which yields a bound which is off by a factor of $\log \log T$ from the expected sharp bound. To simplify the discussion, let us assume that $\ck^2,(\ck')^2 \neq 1$. For $x=T^{1/(\log \log T)^4}$ one expects that $P(j,\ck)$ has a Gaussian limiting distribution over $T<t_j \le 2T$ with mean zero and variance $\sim \frac14\log \log T$.  Consequently for most $T < t_j \le 2T$, $|P(j,\ck)| < (\log \log T)^{3/2}$ and using the Taylor expansion for the exponential function we have for these $j$'s that $ \exp(\pm P(j,\ck))$ is closely approximated by a Dirichlet polynomial of length $\le T^{\varepsilon}$. For $T<t_j \le 2T$ with $|P(j,\ck)|, |P(j,\ck')| < (\log \log T)^{3/2}$ we apply \eqref{eq:elementary} with $M=(\log x)^{1/4} \exp(P(j,\ck))$ and $M'=(\log x)^{1/4} \exp(P(j,\ck'))$, expand $M(M')^{-1},M'M^{-1}$ in terms of short Dirichlet polynomials, then use non-negativity to bound the sum of the fractional moment over these $j$'s in terms of sums of mollified first moments over all $T< t_j \le 2T$. To handle the contribution to the fractional moment from the remaining $T<t_j\le 2T$, we apply \eqref{eq:elementary} with $M=M'=1$ and multiply the right-hand side of the resulting inequality by $(P(j,\ck)/(\log \log T)^{3/2})^{2k}+(P(j,\ck')/(\log \log T)^{3/2})^{2k}$, which is $\ge 1$ for these $j$'s, with $k=\lfloor (\log \log T)^{3/2} \rfloor$, say. We then extend the sum on the right-hand side to all $T< t_j \le 2T$ and evaluate it using the twisted first moment.  Combining the bounds from both cases leads to the estimate \eqref{eq:weakbd} with $x=T^{1/(\log \log T)^4}$. To go further, we use the iterative construction of Radziwi{\l}{\l} and Soundararajan \cite{RadziwillSoundararajan:2015a}, which yields the key inequalities given in Lemma \ref{lem:keyineq}. This enables us to take $x$ to be a small, fixed power of $T$. Applying Lemma \ref{lem:keyineq} and proceeding similarly to before we obtain the optimal bound \eqref{eq:fractional-moment-bd}.

\subsection{Applications to quadratic forms}
We now state an application of Theorem \ref{thm:mainresult} to counting quadratic forms. This is in the spirit of \cite[Theorem 3]{Patterson:1975}.

Let $d_0, d \equiv 1\bmod 4$ be \emph{different} negative fundamental discriminants. Fix a positive definite integral binary quadratic form \begin{equation}p_0(x,y)=a_0x^2+b_0xy+c_0y^2 \textrm{ with } d_0=b_0^2-4a_0c_0.
\end{equation} Let $p(x,y)=ax^2+bxy+cy^2$ be any integral positive definite binary quadratic form of discriminant $d$. Then a straightforward computation shows that the distance between the corresponding Heegner points satisfies
\begin{equation} \label{eq:distance}
  4u\left(\tfrac{-b_0+i\sqrt{\abs{d_0}}}{2a_0},\tfrac{-b+i\sqrt{\abs{d}}}{2a}\right)+2= \frac{2}{\sqrt{\abs{dd_0}}}(2(ac_0+a_0c)-bb_0).
\end{equation}
The quantity
$2(ac_0+a_0c)-bb_0$
is sometimes referred to as the codiscriminant of $p$ and $p_0$.

Arguing as in \cite[Theorem 3]{Patterson:1975} and using Theorem \ref{thm:mainresult} leads to the following result:
\begin{theorem} Let $d,d_0$, $(a_0,b_0,c_0)$ be as above with $d\neq -3$. Then the quadratic form counting function
  \begin{equation} n_d(x) =\#\left\{(a,b,c)\in \Z^3 \left\vert \begin{array}{l}a>0\\ b^2-4ac=d\\0<2(ac_0+ca_0)-bb_0\leq x\end{array}\right.\right\}
  \end{equation}
  satisfies

  \begin{equation}
    n_d(x)=\frac{6h(d)}{\sqrt{\abs{dd_0}}}x
    +
    O\left(\frac{x^{2/3}}{(\log x)^{1/6}}\right  ).
  \end{equation}
\end{theorem}
We note that the restriction to $d\neq -3$ is not essential; if $d=-3$ the main term should simply be divided by 3 to account for a non-trivial stabilizer of the corresponding Heegner points.

A variation of the preceding argument leads to an additional application
to counting $\slz$ equivalence classes of pairs of quadratic forms.

For $\gamma=(\begin{smallmatrix}e&f\\g&h\end{smallmatrix})\in\slz$
we write $p^{\gamma}(x,y)=p(e x+f y,g x+h y)$.
Two pairs of positive definite binary quadratic forms $(p,p_0)$ and $(q, q_0)$ are \textit{equivalent} if there exists $\gamma \in \slz$
such that $(p^{\gamma},p_0^{\gamma})=(q,q_0)$ and in this case we write
$(p,p_0) \sim (q,q_0)$. The discriminants $d,d_0$ as well as the
codiscriminant $\Delta=bb_0-2ac_0-2ca_0<0$ are invariant under $\sim$. It is well-known and easy to verify that $\sim$ is an equivalence
relation on the set of pairs of positive definite binary quadratic forms with discriminants $d,d_0$ and codiscriminant $\Delta$. Also the number of $\slz$ equivalence classes $h(d,d_0,\Delta)$ of such pairs of forms is finite. The class number $h(d,d',\Delta)$ has been previously studied
in the work of Hardy and Williams \cite{HardyWilliams:1989},
Morales \cite{Morales:1991},
and more recently by Bir\'o \cite{Biro:2024}.
In particular, for $(dd',\Delta)=1$ and $h(d,d',\Delta)\neq 0$
Hardy and Williams \cite{HardyWilliams:1989} proved that
\[
  h(d,d',\Delta)
  =\sum_{e|(\Delta^2-dd')/4} \bigg( \frac{d}{e}\bigg),
\]
where $(\frac{d}{e})$ is the Kronecker symbol.

Recall the formula for the distance between Heegner points \eqref{eq:distance}. For each pair of hyperbolic lattice points $\pm \gamma$ counted in $N(X;z_d,z_{d'})$, write $z_{d_0}=\pm \gamma z_{d'}$. It follows that the associated triple $(z_d,z_{d_0}, u(z_d,z_{d_0}))$ corresponds to an $\slz$ equivalence class $(p,p_0,\Delta)$ with codiscriminant $\Delta$ satisfying
\[
  0<-\Delta \le \frac{\sqrt{|dd'|}}{2} X.
\]
Therefore, summing over all Heegner points of discriminants $d,d'$ modulo $\slz$ we have that
\[
  \sum_{z_d,z_{d'}} \frac12 N(X;z_d,z_{d'})=\sum_{0< -\Delta \le \frac{\sqrt{|dd'|}}{2} X} h(d,d',\Delta).
\]
Applying Theorem \ref{thm:mainresult} leads to the following result:

\begin{theorem} Let $d,d'<-3$ be different squarefree discriminants. Then
  \[
    \sum_{0<-\Delta \le x} h(d,d',\Delta)=\frac{6h(d)h(d')}{\sqrt{\abs{dd'}}}x
    +
    O\left(\frac{x^{2/3}}{(\log x)^{1/6}}\right  ).
  \]
\end{theorem}

\subsection{A second moment estimate for the error term}

For any two points $z,w$ we consider the second moment of the error term. Chamizo \cite{Chamizo:1996a}, using his large sieve for Riemann surfaces \cite{Chamizo:1996}, obtained an upper bound of the form
\begin{eqnarray} \label{eq:secondmomentconj}
  \frac{1}{X}\int_{X}^{2X} \left|  N(x;z, w)  - \frac{2 \pi x}{\vol{\GmodH}}  \right|^2 dx \ll X  (\log X)^a
\end{eqnarray}
with $a=2$ (and proved an analogous result for any cofinite Fuchsian group).
This second moment supports Conjecture \ref{mainconjecture} on average.
The power of $\log X$ can in fact be reduced to $a=1$ \cite{Cherubini:2018} using techniques originating in the work of Cram\'er \cite{Cramer:1922}.
Phillips and Rudnick \cite[Section 3.8]{PhillipsRudnick:1994}
noticed that numerical data for Fermat groups indicates that for $z=w$ we should be able to take $a=0$. As a corollary of Theorem \ref{thm:fracmoments} we obtain the following improvement over Heegner points:

\begin{corollary} \label{corollary1.5}
  Let $\G = \slz$ and let $z_d, z_{d'}$ be Heegner points of distinct squarefree discriminants $d,d'<0$. Then
  \begin{align}
    \frac{1}{X}\int_{X}^{2X} \left|N (x;z_d, z_{d'})  - \frac{2 \pi x}{\vol{\GmodH}}\right|^2 dx \ll X  (\log X)^{3/4}.
  \end{align}

\end{corollary}

\subsection{Structure of the paper:} The paper is organized as follows:
In Section \ref{S2} we set notation and recall relevant background. In Section \ref{sectiontwisted} we state and  prove  the twisted first moment, and in Section \ref{sectionfractional} we prove Theorem \ref{thm:fracmoments}. We deduce Theorem \ref{thm:mainresult} from Theorem \ref{thm:fracmoments} in Section \ref{finalproof}. We finally prove Corollary \ref{corollary1.5} in Section \ref{secondmomenterror}.

\section*{Acknowledgements}
The authors are grateful to Chung-Hang Kwan, Min Lee, Bingrong Huang, Yiannis Petridis and Ze{\'e}v Rudnick for very useful discussions in relation to this project.

This material is based upon work supported by the Swedish Research Council
under Grant no. 2021-06596 while the authors were in residence at Institut Mittag--Leffler in Djursholm, Sweden, during the programme \lq Analytic Number Theory\rq{} in
January-April of 2024.
During the preparation of this work M.R. was supported by Grant
DFF-3103-0074B from the Independent Research Fund Denmark.
D. C. was supported by the Hellenic Foundation for Research and Innovation (H.F.R.I.) under the “3rd Call for H.F.R.I. Research Projects to support Faculty Members \& Researchers” (Project Number: 25622).
G.C.~is member of the INdAM group GNSAGA.

\section{Preliminaries}\label{S2}

We begin by setting up some notation and recalling some results that we will need later.

\subsection{Heegner points}\label{sec:heegner}
Recall that a discriminant $d=b^2-4ac \equiv0,1 \pmod 4$ is called fundamental if it cannot be written as $d=f^2d'$ for $d'$ a discriminant and $f>1$. For simplicity, we will always assume that $d \equiv 1\pmod 4$.

Consider $z\in \H$
a Heegner point
with negative fundamental
discriminant
$d<0$, i.e. \begin{equation}
  z=\frac{-b+\sqrt{d}}{2a}
\end{equation}
where $az^2+bz+c$
is an integer quadratic form with $a>0$ and discriminant $d$. Here (and elsewhere) we choose the root such that $\sqrt{d}=i\sqrt{\abs{d}}$.
The ring of integers of the imaginary quadratic field $K=\Q(\sqrt{d})$ is denoted by $\mathcal{O}_K$. The class group of $K$ is denoted by $\Cl_K$,
and its order by $h(d)$. The Heegner points of discriminant $d$ modulo $\slz$ are in 1-1 correspondence with the ideal class group of $K$, and we denoted by $z_{\a}$ the image of $\a\in \Cl_K$ under such a correspondence. For $d$ a fundamental discriminant, there are $h(d) \asymp |d|^{1/2 \pm\varepsilon}$ Heegner points $\{ z_d\}$ of discriminant $d$ modulo $\slz$ and Duke \cite{Duke:1988} proved these points equidistribute in $\slz\backslash\mathbb H$ with respect to hyperbolic measure as $d \rightarrow -\infty$.

\subsection{Class group \texorpdfstring{$L$-functions}{L-functions}}
Let $\widehat{\Cl}_K$ be the group of characters of $\Cl_K$. For $\ck\in \widehat{\Cl}_K$  Hecke \cite{Hecke:1920} defined the corresponding class group $L$-function for $\Re(s)>1$ by    \begin{equation}L_K(s,\ck)=\sum_{0\neq\mathfrak{a}\subseteq \mathcal O_K}\frac{\ck(\mathfrak a)}{N(\mathfrak a)^s}=\prod_{\mathfrak p}(1-\ck(\mathfrak p)N(\mathfrak p)^{-s})^{-1},\end{equation} where the sum is over all  non-zero integral ideals of $\mathcal{O}_K$. This admits meromorphic continuation to $s\in \C$ and satisfies the functional equation
\begin{equation}\label{functional-equation-class-group-L-function}\Lambda_K(s,\ck)=\Lambda_K(1-s, \ck),\end{equation} where $\Lambda_K(s,\ck):=\left(\frac{\abs{d}}{(2\pi)^2}\right)^{s/2}\Gamma(s)L_K(s,\ck)$ is the completed $L$-function.
The continued function $L_K(s,\ck)$ is entire except if $\ck=1$. In the case
$\ck=1$
it has a simple pole
at $s=1$ with residue $2\pi h(d)/(\abs{\mathcal O_K^*}\sqrt{\abs{d}})$.

\subsection{Theta series} \label{sec:theta-series}
The above properties of the class group
$L$-function
can be deduced from the properties of the corresponding theta series
\begin{equation}\label{eq:theta series Fourier expansion}
  f_\ck(z)=\delta_{\ck=1}\frac{\abs{\Cl_K}}{\abs{\Ocal_K^*}} +\sum_{0\neq  \a\subseteq \mathcal{O}_K} \ck(\a) e(N(\a)z)=\sum_{n=0}^\infty\lambda_\ck(n)e(nz),
\end{equation}
where, for $n\geq 1$ we have  \begin{equation}
  \lambda_{\ck}(n)=\sum_{\substack{0\neq \a\subseteq \mathcal{O}_K\\ N(\a)=n}} \ck(\a)
\end{equation}
and by Dirichlet's class number formula \cite[Equation (22.59)]{IwaniecKowalski:2004}
\begin{equation}\label{eq:zero-fourier-coefficient}
  \lambda_\ck(0)=\delta_{\ck=1}\frac{\sqrt{\abs{d}}L(1,\chi_d)}{2\pi}.
\end{equation} Here $\chi_d(n)=\left(\tfrac{d}{n}\right)$ is the field character, i.e. $\left(\tfrac{d}{n}\right)$ is the Kronecker symbol. We note that $\chi_d$ is a primitive quadratic character modulo $\abs{d}$ with sign equal to the sign of $d$.

By the ideal theory for imaginary quadratic fields \cite[page 57]{IwaniecKowalski:2004} we find that
\begin{equation}\label{eq:basics of lambda_chi}
  \lambda_\ck(n)\in \R,\quad  \abs{\lambda_\ck(n)}\leq d(n),\quad \textrm{ and }  \lambda_\ck(n)=\pm 1 \textrm{ if }n\vert d.
\end{equation}

The theta series $f_\ck$ is a modular form of weight 1 for $\G_0(\abs{d})$ with character $\chi_d$, and with Hecke eigenvalues $\lambda_\ck(n)$.
For $\ck$ not real $f_\ck$ is   a primitive \emph{cusp form}.
We refer to \cite[Section 12.4]{Iwaniec:1997} for additional details.
Clearly the $L$-function of $f_\ck$
coincides with the corresponding class group $L$-function, i.e.~$L(s, f_\ck)=L_K(s,\ck)$ and it has an Euler product
\begin{equation}\label{eq:euler-product-class-group-L}L(s, f_\ck)=\prod_{p}(1-\lambda_\ck(p)p^{-s}+\chi_d(p)p^{-2s})^{-1}=\prod_{p}\prod_{r=1}^2(1-\alpha^{(r)}_\ck(p) p^{-s})^{-1}.\end{equation}

\subsubsection{Real characters} \label{sec:real-characters}The subgroup  of real characters  is denoted by $\mathcal G_K\subseteq \widehat{\Cl}_K$.
If $\ck\in \mathcal G_K$
it can be realized as a \emph{genus character} as follows (see e.g. \cite[page 54 ff.]{Siegel:1980}): We may factor $d=d_1d_2$ as a product of discriminants. This corresponds to a factorization
\begin{equation}\chi_{d}=\left(\frac{d}{\cdot}\right)=\left(\frac{d_1}{\cdot}\right)\left(\frac{d_2}{\cdot}\right)=\chi_{d_1}\chi_{d_2}.
\end{equation}
For such a factorization Gauss defined a real class group character $\ck_{d_1,d_2} \in \widehat{\Cl}_{K}$  by setting
\begin{equation}
  \ck_{d_1,d_2}(\mathfrak p)=\left(\tfrac{d_i}{N(\mathfrak p)}\right),
\end{equation}
where the prime ideal $\mathfrak p \nmid (d_i)$. If $\mathfrak p \mid (d)$ there is a unique such $i$ since $d$ is squarefree, and if $\mathfrak p \nmid (d)$ we have $(\frac{d_1}{N(\mathfrak p)})=(\frac{d_2}{N(\mathfrak p)})$ by the factorization theory of ideals in
quadratic~fields.
By extending to all fractional ideals by multiplicativity, this defines a character of $\widehat{\Cl}_{K}$,
and all real characters can be realized in this way
for some factorization of $d$ into discriminants. The genus characters
form a subgroup of order $2^{k-1}$
where $k$ is the number of prime factors of $d$. Since factorizations of $d$ into pairs of discriminants correspond to factorizations of $\abs{d}$ into pairs of positive integers we have a bijection between the set of \emph{unordered} pairs $(q_1,q_2)$ of positive integers
with $q_1q_2=\abs{d}$
and $\mathcal G_K$ given by
$(q_1,q_2)\mapsto \ck_{D(q_1),D(q_2)}$
where $D(q)=(-1)^{(q-1)/2}q$.
Here $(1, \abs{d})$ corresponds to the trivial character.

For $\ck=\ck_{d_1,d_2}$ we have Kronecker's factorization formula \cite[Theorem 4]{Siegel:1980}
\begin{equation}\label{eq:Kronecker-factorization}L_K(s, \ck_{d_1,d_2})=L(s,\chi_{d_1})L(s,\chi_{d_2}).\end{equation}
This shows
\begin{equation}\label{eq:Fourier-coefficients when genus}\lambda_{\ck_{d_1,d_2}}(n)=\sum_{m\vert n}\chi_{d_1}(m)\chi_{d_2}(\tfrac n m)={\chi_{d_1}}*{\chi_{d_2}}(n).\end{equation}
We note that
$\lambda_{\ck_{d_1,d_2}}(n)$ is the $n$th Hecke eigenvalues of a certain weight 1 Eisenstein series, see \cite[Section 3.2 and Equation (4.13)]{Young:2019}. More precisely, we find, using the explicit computations in \cite[Proposition 4.13]{Young:2019}, Dirichlet's class number formula \cite[Equation (22.59)]{IwaniecKowalski:2004}, a Gauss sum computation \cite[Theorem 9.17]{MontgomeryVaughan:2007}, \eqref{eq:theta series Fourier expansion} and \eqref{eq:Fourier-coefficients when genus}, that
\begin{equation}\label{eq:Eisenstein characters to theta series}
  E^{*}_{\chi_{d_1},\chi_{d_2}}(z,\tfrac 12)=\sqrt{4\pi y}f_{\ck_{d_1,d_2}}(z)
\end{equation}
where $E^{*}_{\chi_{d_1},\chi_{d_2}}(z,s)$ is the \lq completed\rq{}  character Eisenstein series as defined in \cite[Section 4.1]{Young:2019}.

\subsection{\texorpdfstring{Hecke--Maass $L$-functions}{Hecke--Maass L-functions}}
Consider an orthonormal system of Hecke--Maass cusp forms $\{\phi_j\}$
for the modular group $\slz$.
Hence, each $\phi_j$ is an eigenfunction of the hyperbolic Laplacian
with eigenvalue
\begin{equation}
  \lambda_j = \bigl(\tfrac12+it_j\bigr)\bigl(\tfrac12-it_j\bigr)\geq 0.
\end{equation}
We have (see e.g. \cite[Chapter~11]{Iwaniec:2002}) that either $\lambda_j=0$
or $\lambda_j>1/4$ (in which case we select $t_j\in\RR^+$).
Moreover, $\phi_j$ has a system
$\{\lambda_j(n)\}$ of Hecke eigenvalues satisfying the Hecke relations
\begin{equation}\label{eq:Hecke-relations}
  \lambda_j(m)\lambda_j(n) = \sum_{d|(m,n)} \lambda_j\left(\tfrac{mn}{d^2}\right).
\end{equation}
Associated to $\phi_j$ there is an $L$-function given for $\Re(s)>1$ by
\begin{equation}
  L(s,\phi_j) = \sum_{n\geq 1}\frac{\lambda_j(n)}{n^s},
\end{equation}
which admits an Euler-factorization
\begin{equation}\label{0404:eq001}
  L(s,\phi_j)= \prod_{p} (1-\lambda_j(p)p^{-s}+p^{-2s})^{-1}=\prod_{p} \prod_{r=1}^2(1-\alpha^{(r)}_j(p)p^{-s})^{-1}.
\end{equation}
This $L$-function admits analytic continuation to $\CC$ and satisfies a functional equation
relating $s$ and $1-s$, which
depends on the parity of $\phi_j$:
define $\epsilon_j\in\{\pm 1\}$ by the relation $\phi_j(-\overline z)=\eps_j\phi(z)$;
then we say that $\phi_j$ is even (resp.~odd) if $\eps_j=1$ (resp.~$\eps_j=-1$).
Also, define the local factor at infinity to be
\begin{equation}
  L_\infty(s,\phi_j) = \prod_{\pm}\Gamma_\RR\left(s+\frac{1-\eps_j}{2}\pm it_j\right),
\end{equation}
where $\Gamma_\RR(s):=\pi^{-\tfrac{s}{2}}\Gamma(\frac s2)$, and the completed $L$-function
$\Lambda(s,\phi_j) = L(s,\phi_j) L_\infty(s,\phi_j)$.
Then we have (see e.g.~\cite[\S 2.2]{BlomerFouvryKowalskiMichelEtAl:2023}\footnote{Note that \cite{BlomerFouvryKowalskiMichelEtAl:2023}  has a typo in the formula for $L_\infty(s,\phi_j)$; they have $s-$ instead of $s+$.})
\begin{equation}
  \Lambda(s, \phi_j)=\eps_j\Lambda(1-s, \phi_j).
\end{equation}
In particular, this implies that for odd forms the central value $L(\frac 12,\phi_j)$ vanishes.

\subsection{Character twists of Hecke--Maass forms}
Let $\phi_j$ be a Hecke--Maass form of full level, and let $\chi$ be a primitive Dirichlet character modulo $q$.
We consider the character twist
\begin{equation}
  L(s,\phi_j\otimes\chi) = \sum_{n\geq 1} \frac{\lambda_j(n)\chi(n)}{n^s} = \prod_{p}(1-\lambda_j(p)\chi(p)p^{-s}+\chi^2(p)p^{-2s})^{-1}
\end{equation}
for $\Re(s)>1$. The associated Gamma factors are
(see e.g.~\cite[Lemma 2.1]{BlomerFouvryKowalskiMichelEtAl:2023})
\begin{equation}
  L_\infty(s,\phi_j\otimes \chi) = q^s\prod_\pm \Gamma_\R\left(s+\frac{(1-\eps_j\chi(-1))}{2}\pm it_j\right).
\end{equation}
Then by \cite[Lemma 2.1]{BlomerFouvryKowalskiMichelEtAl:2023} the product
$
  \Lambda(s, \phi_j\otimes\chi):= L_\infty(s,\phi_j\otimes \chi)L(s,\phi_j\otimes \chi)
$
admits analytic continuation to $\C$ and satisfies the functional equation
\begin{equation}\label{functional-equation-character-twist}
  \Lambda(s,\phi_j\otimes\chi)=\epsilon(\phi_j\otimes \chi)\Lambda(1-s,\phi_j\otimes \overline{\chi}).
\end{equation}
Here the root number is given by
$\epsilon(\phi_j\otimes\chi)=\epsilon_j\chi(-1)\epsilon_\chi^2$, where%
\begin{equation}\label{NormalizedGaussSum}
  \epsilon_\chi=\frac{1}{\sqrt{q}} \sum_{h \bmod {q}}  \chi(h) e\Bigl(\frac{h}{q} \Bigr)
\end{equation}
is the normalized Gauss sum.
Since $\epsilon_\chi$ lies on the unit circle and satisfies $\overline{\epsilon_\chi}=\chi(-1)\epsilon_{\overline{\chi}}$  (see \cite[Theorems 9.7, 9.5]{MontgomeryVaughan:2007}) we find that if $\chi$ is real then  $\epsilon_\chi^2=\chi(-1)$ and therefore
$
  \epsilon(\phi_j\otimes\chi)=\epsilon_j.
$
It follows that $L(\tfrac 12,\phi_j\otimes\chi)=0$ for every \emph{real} primitive character $\chi$ and every $\phi_j$ odd.

\subsection{Rankin--Selberg convolution} \label{sec:Rankin-Selberg}
Given a class group character $\ck$ and its associated theta series $f_\ck$
as in Section~\ref{sec:theta-series}, we wish to study the Rankin--Selberg convolution $L$-function
\begin{equation}
  L(s,\phi\times f_\ck) = L(2s,\chi_d) \sum_{n\geq 1} \frac{\lambda(n)\lambda_{\ck}(n)}{n^s},
  \qquad\Re(s)>1. \end{equation} Here either $\phi=\phi_j$ for $\phi_j$ a cusp form of full level or $\phi=E_t=E(\cdot,\tfrac12+it)$ is the  Eisenstein series of full level.  The number $\lambda(n)$ is the $n$th Hecke eigenvalue of $\phi$. Note that in the case of Eisenstein series $\lambda(n)=\lambda_t(n)=\sum_{ab=n}a^{it}b^{-it}$. Let $t$ be the spectral parameter of $\phi$.

The Rankin--Selberg convolution admits an Euler product given by (compare \cite[Lemma 1.6.1]{Bump:1997})
\begin{equation} \label{eq:eulerproduct}
  L(s,\phi_j\times f_\ck)=\prod_p\prod_{r_1,r_2=1}^2 (1-\alpha^{(r_1)}_j(p)\alpha^{(r_2)}_{\ck}(p)p^{-s}))^{-1},
\end{equation}
where $\alpha_j^{(r)}(p)$ and $\alpha_\ck^{(r)}(p)$ are as in \eqref{0404:eq001} and \eqref{eq:euler-product-class-group-L}, respectively.
It follows that its logarithm has a Dirichlet series representation
\begin{equation}\log L(s,\phi_j\times f_\ck)=\sum_p\sum_{n=1}^\infty\frac{1}{n}\sum_{r_1,r_2=1}^2\left(\alpha_j^{(r_1)}(p)\alpha^{(r_2)}_{\ck}(p)\right)^n \frac{1}{p^{ns}}
\end{equation}
We note for later that the first term in the $n$-sum equals
\begin{equation}\label{eq:dominant-term-Euler-prod}
  \lambda_\ck(p)\lambda_j(p)p^{-s},
\end{equation}
while the $n=2$ term equals
\begin{equation}\label{eq:secondary-term-Euler-prod}\frac{(\lambda_j(p^2)-1)(\lambda_\ck(p^2)+\chi_d(p)-2)}{2}p^{-2s}.
\end{equation}
Returning to the situation
of a general $\phi$ consider the gamma factor
\begin{equation}
  L_\infty(s,\phi\times f_\ck) = \left(\frac{\abs{d}}{4\pi^2}\right)^s\prod_{\pm} \Gamma(s\pm it)
\end{equation}
and define the completed $L$-function
$\Lambda(s, \phi\times f_\ck) = L_\infty(s,\phi\times f_\ck)L(s,\phi_j\times f_\ck)$. Then $\Lambda(s, \phi\times f_\ck)$ admits meromorphic continuation and satisfies the functional equation
\begin{equation}\label{functional-equation-RS}
  \Lambda(s, \phi\times f_\ck) = \Lambda(1-s,\phi\times f_\ck).
\end{equation} It is entire except in the case where $\ck=1$ and $\phi=E_t=E(\cdot,\tfrac12+it)$ is Eisenstein; in this case it has poles at $1\pm it$ and $\pm it$.

To see this consider first the case where $\phi=E_t$ is Eisenstein. Then by a local computation using \cite[Lemma 1.6.1]{Bump:1997} we see that
\begin{equation}\label{eq:factorisation when Eisenstein}L(s,E_t\times f_\ck) = L_K(s-it, \ck)L_K(s+it, \ck)\end{equation}
and \eqref{functional-equation-RS} now follows from the functional equation of the class group $L$-function \eqref{functional-equation-class-group-L-function}. If $\ck=\ck_{d_1,d_2}$ is a genus character it factors further into 4 degree 1 $L$-functions
\begin{equation}\label{eq:double-factorization}L(s,E_t\times f_{\ck_{d_1,d_2}})=L(s-it,\chi_{d_1})L(s-it,\chi_{d_2})L(s+it,\chi_{d_1})L(s+it,\chi_{d_2}),
\end{equation}
by the Kronecker factorization \eqref{eq:Kronecker-factorization} This proves the claim about the poles.

Assume next that $\phi_i$ is cuspidal. Then \eqref{functional-equation-RS} follows from \cite[page 136 and Appendix]{KowalskiMichelVanderKam:2002}.
If $\ck=\ck_{d_1,d_2}$ is a genus character, a local computation, again using \cite[Lemma 1.6.1]{Bump:1997}, shows that
\begin{equation} \label{eq:RS-factorization-genus-characters}
  L(s,\phi_j\times f_{\ck_{d_1,d_2}}) = L(s,\phi_j\otimes\chi_{d_1})L(s,\phi_j\otimes\chi_{d_2}).
\end{equation} Note that in this case the functional equation \eqref{functional-equation-RS} follows also from the functional equation of the character twists \eqref{functional-equation-character-twist}
using
that $\chi_{d_i}$ is a primitive character of conductor $\abs{d_i}$, and that $d_1d_2=d<0$, so one of $\chi_{d_1}$, $\chi_{d_2}$ is
even and the other is odd.

\subsection{Waldspurger's formula} \label{waldspurgersubsection}
Let $\phi_j$ be a Hecke--Maass cusp form.  The Waldspurger--Zhang formula \cite{Waldspurger:1981, Zhang:2001a, Zhang:2001, Zhang:2004} (see also \cite[p. 647]{HarcosMichel:2006},  \cite[Eq. (3.3)]{LiuMasriYoung:2013}) gives that
\begin{equation}\label{eq:Waldspurger}
  \frac{L(\tfrac{1}{2},\phi_j\times f_\ck)}{L(1, \sym^2 \phi_j)}=c_d^{-1}\abs{\sum_{\a\in \Cl_K} \ck(\a) \phi_j(z_{\a})}^2.
\end{equation}
Here $c_d=c\sqrt{\abs{d}}$ where $c$ is a
positive
absolute constant. There is some disagreement in the literature about the precise value of $c$ (See the discussion in \cite[p. 1339-1340]{BlomerCorbett:2022}) but it is not important for our purposes.

Using the orthogonality of characters on finite groups \cite[Thm 4.4]{MontgomeryVaughan:2007} we find from \eqref{eq:Waldspurger} that
\begin{equation}\label{eq:Waldspurger2}
  \frac{c_d}{h(d)} \sum_{\ck\in \widehat{\Cl}_K}
  \frac{L(\tfrac{1}{2},\phi_j\times f_\ck)}{L(1, \sym^2 \phi_j)}=\sum_{\a\in \Cl_K} \abs{\phi_j(z_{\a})}^2.
\end{equation}
It follows, since all terms on both sides are non-negative, that for any Heegner point of discriminant $d$ we have
\begin{equation} \label{eq:phibd}
  \abs{\phi_j(z)}\leq\left(
  \frac{c_d}{h(d)} \sum_{\ck\in \widehat{\Cl}_K}
  \frac{L(\tfrac{1}{2},\phi_j\times f_\ck)}{L(1, \sym^2 \phi_j)} \right)^{1/2}\ll_d
  \sum_{\ck\in \widehat{\Cl}_K}  \frac{L(\tfrac{1}{2},\phi_j\times f_\ck)^{1/2}}{L(1, \sym^2 \phi_j)^{1/2}}
\end{equation}
where we have used that the square root is concave and nonnegative on the non-negative reals (or simply by observing directly that $\sqrt{a+b}\leq\sqrt{a}+\sqrt{b}$).

\section{Twisted first moment} \label{sectiontwisted}

In this section we prove an asymptotic formula for the first moment
of $L(\frac12,\phi_j\times f_\ck)$ twisted by
an additional Hecke eigenvalue of $\phi_j$, and normalized by the standard harmonic factor $L(1,\sym^2\phi_j)$.

We use a non-negative test function $h(t)$ that localizes
around $\abs{t}\asymp T$, for some real parameter $T$ tending to infinity.
In particular, our proof works for the test function
\begin{equation}\label{def:h}
  h(t) = e^{-\tfrac{t^2}{T^2}} \frac{t^{2M}}{T^{2M}} \prod_{k=1}^{M} \left(\frac{t^2+(k-\tfrac 12)^2}{T^2}\right)^2,
\end{equation}
where $M$ is a sufficiently large positive integer
(independent of $T$).
The main result of this section is the following theorem:

\begin{thm}\label{thm:Twisted-first-moment} Let $\ck\in \widehat{\Cl}_K$.
  Fix $\varepsilon\in(0,\frac{1}{100})$. For $M\geq M(\varepsilon)$ and $T\gg 1$, consider the test function $h(t)$
  given in \eqref{def:h}.
  For any $\varepsilon_0>0$, as $T\to\infty$ we have, uniformly in $1\leq \ell < T^{2-\varepsilon_0}$,  that
  \begin{equation}
    \sum_{t_j>0} \frac{L(\tfrac12,\phi_j\times f_\ck)\lambda_{j}(\ell)}{L(1,\sym^2\phi_j)}h(t_j)
    =
    C_{M,d} \frac{\lambda_{\ck}(\ell)}{\sqrt{\ell}} T^2 + O_{\varepsilon_0,\varepsilon,d,M}(T^{1+\varepsilon}) ,
  \end{equation}
  where
  \begin{equation}
    C_{M,d}=\frac{2L(1,\chi_d)}{\pi^2}\int_{0}^\infty t^{6M+1}e^{-t^2}dt.
  \end{equation}
\end{thm}

We note that Theorem \ref{thm:Twisted-first-moment} is similar in spirit to \cite[Theorem 1.9]{HoffsteinLeeNastasescu:2021}. Our proof is simpler and the error terms are better, but our test function is more restricted. Also \cite{HoffsteinLeeNastasescu:2021} has the additional simplifying restriction that the level of $f_\ck$ should divide the level of the family we are averaging over.

The rest of the section will be dedicated to proving Theorem \ref{thm:Twisted-first-moment}.

\subsection{Approximate functional equation}
Let $\phi=\phi_j$ or $\phi=E(\cdot,\tfrac12+it)$ and let $\lambda(n)$ be the Hecke eigenvalues of $\phi$. By a small abuse of notation we let $t$ be the spectral parameter of $\phi$ i.e. $t=t_j$ if $\phi=\phi_j$ and $t=t$ if $\phi=E(\cdot,\tfrac12+it)$.
Then Section \ref{sec:Rankin-Selberg} and \cite[Theorem 5.3]{IwaniecKowalski:2004}
provide
an approximate functional equation
\begin{equation}\label{eq:approximate functional equation}
  L(\tfrac 12,\phi\times  f_\ck)
  =
  2\sum_{n\geq 1} \frac{\lambda(n)\lambda_{\ck}(n)}{\sqrt{n}}
  W(n,t) -R.
\end{equation}
Here
\begin{equation}\label{def:Wnt}
  W(y,t) = \frac{1}{2\pi i}\int_{(3)} \gamma(s,t)
  L(1+2s,\chi_d) \frac{e^{s^2}}{s} \left(\frac{\abs{d}}{4\pi^2y}\right)^{s} ds
\end{equation}
where
\begin{equation}
  \gamma(s,t)=\frac{\Gamma(\tfrac 12+s+it)\Gamma(\tfrac 12+s-it)}{\Gamma(\tfrac 12+it)\Gamma(\tfrac 12-it)}
\end{equation}
and $R=0$, unless $\phi=E(\cdot,\tfrac{1}{2}+it)$ and $\ck=1$
when
the poles of the completed version of \eqref{eq:double-factorization} give the additional term
\begin{align}
  R & =L(1,\chi_d)\sum_\pm L(1\pm 2it,\chi_d)\zeta(1\pm 2it)\left(\frac{\abs{d}}{4\pi^2}\right)^{\tfrac{1}{2}\pm it}\frac{\Gamma(1\pm 2it)}{\abs{\Gamma(1/2+it)}^2}\frac{e^{(\tfrac{1}{2}\pm it)^2}}{\tfrac{1}{2}\pm it}                \\
    & \quad+L(0,\chi_d)\sum_\pm L(\pm 2it,\chi_d)2\zeta(0)\zeta(\pm 2it)\left(\frac{\abs{d}}{4\pi^2}\right)^{-\tfrac{1}{2}\pm it}\frac{\Gamma(\pm 2it)}{\abs{\Gamma(1/2+it)}^2}\frac{e^{(-\tfrac{1}{2}\pm it)^2}}{-\tfrac{1}{2}\pm it}.
\end{align}
Here  we have assumed that $t\neq 0$. There is a similar formula for $R$ when $t=0$ coming from the double pole of $L(s,E(\cdot, \tfrac{1}{2})\times f_1)$, but we will not need it.
Let us record the following bound for $R$ in the case $\phi=E(\cdot,\tfrac{1}{2}+it)$ and $\ck=1$
\begin{equation} \label{eq:Rbd}
  R \ll  (1+ \abs{t})^{-1/2+\varepsilon} e^{-t^2} \max(1,1/|t|)
\end{equation}
for nonzero $t \in \mathbb R$, where we have used standard bounds for Dirichlet $L$-functions at
the edge of the critical strip.

The properties that we need about
the smoothing function $W$ defined in \eqref{def:Wnt}
are summarized in the following lemma:
\begin{lemma}\label{0112:lemma}
  Let $y\geq 1$ and let $W(y,t)$ be as in \eqref{def:Wnt}.
  For all $\eta\in(0,1/2)$ and nonzero $t \in \mathbb R$ we have
  \begin{equation}\label{2911:eq001}
    W(y,t)=L(1,\chi_d)+O_{d,\eta}\biggl(\left(\frac{y}{|t|^2}\right)^{\!\!\eta\;}\biggr).
  \end{equation}
  For any $A\geq 0$, $W(y,t)$ is holomorphic in $t$ in the strip $|\Im(t)|\leq A$.
  When $\Im(t)=-A$ we have, for all $B\geq A$,
  \begin{equation}\label{1302:eq001}
    W(y,t) \ll_{A,B,d} \left(\frac{y}{|t|^2}\right)^{-B}.
  \end{equation}
  Moreover, for any $N\geq 0$, $y\geq 1$, and $t\in \R$ with  $\abs{t}\geq 1$  we have
  \begin{equation}\label{eq:W-expansion}
    W(y,t) = \sum_{r=0}^{N-1}\frac{1}{\abs{t}^r}W_r(y,t)
    +O\left(\frac{1}{\abs{t}^N}\right),
  \end{equation}
  where
  \begin{equation}\label{eq:W_r-definition}
    W_r(y,t)=\frac{1}{2\pi i}\int_{(3)}
    P_r(s)L(1+2s,\chi_d) \frac{e^{s^2}}{s} \left(\frac{(\frac 14 +t^2)\abs{d}}{4\pi^2y}\right)^{s} ds
  \end{equation}
  and $P_r$ are explicit polynomials of degree at most $2r$ with $P_0=1$. We have for $t\in \mathbb R$ \begin{equation}\label{eq:Wr-trivial bound}W_r(y,t)\ll_{r,d} (1+\abs{t})^\e.\end{equation}

\end{lemma}
\begin{proof}
  The bound in \eqref{2911:eq001} is proved similarly as in the proof of \cite[Proposition 5.4]{IwaniecKowalski:2004}: first move the line of integration to $\Re(s)=-\eta$ and pick up the residue at $s=0$, which gives the main term $L(1,\chi_d)$. Then on the new contour bound in absolute value combining the rapid decay of $e^{s^2}$ with Stirling's formula to bound the Gamma factors as follows; by using the first term in Stirling's formula \cite[5.11.3]{OlverLozierBoisvertClark:2010} one shows that for $t$ real and $s$ in a vertical strip  $-\tfrac12<\delta\leq \Re(s)\leq B $ we have
  \begin{equation} \label{eq:gammabd}
    \gamma(s,t)\ll \abs{\tfrac12 +s+it}^{\Re(s)}\abs{\tfrac12 +s-it}^{\Re(s)}.
  \end{equation}
  To see that the function is holomorphic, move the line of integration to $\Re(s)=A$ and notice that the integrand has no poles
  when $|\Im(t)|<\Re(s)+1/2$. To prove \eqref{1302:eq001} move the contour to $\Re(s)=B$ and then bound in absolute value.

  To prove \eqref{eq:W-expansion}
  we proceed as follows;  using \eqref{eq:gammabd}
  allows us to use the exponential decay of $e^{s^2}$ in vertical strips to truncate the integral defining  $W(y,t)$ and get that for any $A, \e>0$
  \begin{equation}\label{eq:truncated W}
    W(y,t) = \frac{1}{2\pi i}\int_{(3),\abs{s}\leq {\abs{t}}^\e} \gamma(s,t)
    L(1+2s,\chi_d) \frac{e^{s^2}}{s} \left(\frac{\abs{d}}{4\pi^2y}\right)^{s} ds+O\left(\frac{1}{\abs{t}^A}\right).
  \end{equation}
  The implied constant depends on $A,\e$ and $d$.
  We then need a more precise estimate for $\gamma(s,t)$ as in \cite[Section 5]{Young:2017}:
  For $z,u\in \mathbb C$ with
  $\abs{\mathrm{Arg}(z+u)}$, $\abs{\mathrm{Arg}(z)}$
  both bounded by $\pi-\delta$ for some $\delta>0$ Stirling's formula gives
  \begin{equation}
    \frac{\Gamma(z+u)}{\Gamma(z)}=z^u\left(\sum_{r=0}^{N-1}\frac{p_r(u)}{z^r}+O_N\left(\frac{(1+\abs{u})^{2N}}{\abs{z}^N}\right)\right)
  \end{equation}
  for $\abs{z}$ bounded away from zero, and $u^2/z$ sufficiently small.
  Here $p_0(u)=1$ and $p_r(u)$ is a polynomial of degree at most $2r$.

  This implies that for $\abs{t}\geq 1$ and $\frac{\abs{s}^2}{\abs{t}}$ sufficiently small we have
  \begin{equation}
    \gamma(s,t)=(\tfrac 14+t^2)^s\left(\sum_{r=0}^{N-1} \frac{P_r(s)}{\abs{t}^r}+O_N\left(\frac{(1+\abs{s})^{2N}}{(1+\abs{t})^N}\right)\right)
  \end{equation}
  with  $P_0=1$
  and $P_r$ polynomials of degree at most $2r$. We can now insert this back into \eqref{eq:truncated W}, and extend the integration range again using the exponential decay of $e^{s^2}$, and we arrive at \eqref{eq:W-expansion}.
  Moving the line of integration in
  \eqref{eq:W_r-definition}
  to any $\e>0$ we see that for $y\geq 1$ we have \eqref{eq:Wr-trivial bound}.
\end{proof}

\subsection{Bounding the Kuznetsov transform}
To prove Theorem \ref{thm:Twisted-first-moment} we we will  use the same-sign Kuznetsov trace formula \cite[Theorem 16.3]{IwaniecKowalski:2004}, \cite[Theorem 2.1]{HumphriesKhan:2023}. The integral transform appearing
on the Kloosterman sum side
is defined by
\begin{equation}
  (\EuScript{K}^{+}h)(x) = \frac{2i}{\pi} \int_\RR th(t) J_{2it}(x) \frac{dt}{\cosh(\pi t)},
\end{equation}
where $J_\nu$ is the Bessel function of the first kind of order $\nu$.

Let $h(t)$, $W(n,t)$ be as in \eqref{def:h}, \eqref{def:Wnt} and define \begin{equation}\label{eq:hn-def}
  h_n(t)=h(t)W(n,t).
\end{equation}

\begin{lemma}\label{2911:lemma1}
  For every $n\geq 1$ and all $x>0$ we have
  \begin{equation}\label{2911:eq002}
    (\EuScript{K}^{+}h_n)(x) \ll_{A,B,M,d} \frac{T^{2+2B-2A}x^{2A}}{n^{B}}
  \end{equation}
  where $0\leq A\leq M$ and $B\geq A$.

  If $T^{-100}<x<T^{2-\varepsilon_0}$,
  we have for all $n\geq 1$,
  \begin{equation}\label{2911:eq003}
    (\EuScript{K}^{+}h_n)(x) \ll_{M,\varepsilon_0} T^{-1+\varepsilon}.
  \end{equation}
\end{lemma}

\begin{proof}
  First we recall, for $x>0$, the integral representation for the $J$-Bessel function as
  \cite[8.411.10]{GradshteynRyzhik:2007}
  \begin{equation}
    J_{2it}(x) = \frac{(x/2)^{2it}}{\sqrt{\pi}\,\Gamma(\tfrac12+2it)} \int_{-1}^{1} e^{iux}(1-u^2)^{-1/2+2it} du,\qquad
    \Re(2it)>-\frac 12.
  \end{equation}
  The reciprocal of the Gamma function is entire and the integral is finite
  when $x>0$ and $\Re(2it)>-1/2$.
  This shows that $J_{2it}(x)$ is holomorphic in $t$ in the region $\Im(t)<\tfrac 14$.
  Applying Stirling's asymptotic on the Gamma function,
  we can bound on any fixed line $\Im(t)=-A\leq 0$
  \begin{equation}\label{2811:eq001}
    |J_{2it}(x)| \leq \frac{\sqrt{\pi}\,(\tfrac x2)^{\Re(2it)}}{|\Gamma(\tfrac 12+2it)|} \ll_A \biggl(\frac{x}{1+|t|}\biggr)^{2A}e^{\pi|t|}.
  \end{equation}
  The function $h_n(t)$ is holomorphic in the strip $-A\leq \Im(t)\leq 0$
  and vanishes at the points
  $t=\tfrac{ki}{2}$,
  for all odd values of $k$ in the range $-2M\leq k\leq 2M$,
  balancing the poles from $\cosh(\pi t)^{-1}$
  in the integrand defining $\EuScript{K}^+h_n$ for $-M \le \Im(t) \le 0$.
  Due to the rapid decay of $h(t)$, the $t$-integral defining $\EuScript{K}^{+}h_n$
  converges absolutely in the whole strip
  $-M\leq \Im(t)\leq 0$.
  On the line $\Im(t)=-A$ we find using \eqref{1302:eq001},
  \begin{equation}\label{eq:h_n-bound}
    h_n(t) \ll_{A,B,M,d} e^{-|t|^2/T^2} \frac{|t|^{6M+2B}}{T^{6M}n^{B}},
  \end{equation}
  which combined with \eqref{2811:eq001} gives  \eqref{2911:eq002}.

  Proving \eqref{2911:eq003} is more delicate;  when bounding $(\EuScript{K}^{+}h_n)(x)$
  it is convenient to first remove the contribution to the integral for $t$ small. We do this by using a partition of unity as follows; let $0\leq f_j \leq 1$, $j=0,1$ be fixed even smooth functions satisfying
  \begin{equation}\supp(f_0)\subseteq [-2,2],\quad  \supp(f_   1)\subseteq \mathbb R\backslash[-1,1], \quad f_0(t)+f_1(t)=1.    \end{equation}  We then choose $1\leq X\leq T$ and split
  \begin{equation}
    (\EuScript{K}^{+}h_n)(x)=(\EuScript{K}_0^{+}h_n)(x)+(\EuScript{K}_1^{+}h_n)(x)
  \end{equation}
  where for $j=0,1$
  \begin{equation}(\EuScript{K}_j^{+}h_n)(x)=\frac{2i}{\pi} \int_\RR f_j(t/X)th_n(t) J_{2it}(x) \frac{dt}{\cosh(\pi t)}.
  \end{equation}
  Using the bounds \eqref{2811:eq001}, \eqref{eq:h_n-bound} with $A=B=0$ we see that
  \begin{equation}
    (\EuScript{K}_0^{+}h_n)(x)\ll \int_0^{2X}\frac{t^{6M+1}}{T^{6M}}e^{-t^2/T^2}dt\ll \frac{X^{6M+2}}{T^{6M}}.
  \end{equation}
  We will eventually choose $X=T^{\alpha}$ for some $\alpha<1$, so this contribution can be made small by choosing $M$ large. This takes care of the contribution coming from small $t$.

  To deal with the contribution for $t$ large we note that by \eqref{eq:W-expansion} and \eqref{2811:eq001} we have
  \begin{equation}
    (\EuScript{K}_1^{+}h_n)(x)=\sum_{r=0}^{N-1}(\EuScript{K}_{1,r}^{+}h_n)(x)+O(X^{2-N}),
  \end{equation}
  where \begin{align}
    (\EuScript{K}_{1,r}^{+}h_n)(x)
     & = \frac{2i}{\pi} \int_0^\infty f_1(t/X)t^{1-r}h(t)W_r(n,t) \left(J_{2it}(x)-J_{-2it}(x)\right) \frac{dt}{\cosh(\pi t)}.
  \end{align}
  We now use a different expression for the difference of Bessel functions. By \cite[8.411.11]{GradshteynRyzhik:2007} we find that
  \begin{equation}
    J_{2it}(x)-J_{-2it}(x)=\frac{2}{\pi i}\sinh(\pi t){\int_\R}\cos(x\cosh(u))e(\tfrac{tu}{\pi})du.
  \end{equation}
  By partial integration we find that
  \begin{equation}
    \int_\R\cos(x\cosh(u))e(\tfrac{tu}{\pi})du=\int_{-T^\e}^{T^{\e}}\cos(x\cosh(u))e(\tfrac{tu}{\pi})du+O\left(\frac{1+\abs{t}}{x}e^{-T^\e}\right).
  \end{equation}
  Using \eqref{eq:Wr-trivial bound} we find that $(\EuScript{K}_{1,r}^{+}h_n)(x)$ equals
  \begin{equation}
    \frac{4}{\pi^2} \int_0^\infty f_1(t/X)t^{1-r}h(t)W_r(n,t)\int_{-T^\e}^{T^\e} \cos(x\cosh(u))e(\tfrac{tu}{\pi})du \tanh{\pi t}dt+ O\left(\frac{e^{-T^\e/2}}{x}\right).
  \end{equation}
  We can replace $\tanh(\pi t)$ by 1 at the expense of an error of $O(e^{-X}T^{2+\e})$, so we have to estimate
  \begin{equation}
    \int_0^\infty f_1(t/X)t^{1-r}h(t)W_r(n,t)\int_{-T^\e}^{T^\e} \cos(x\cosh(u))e(\tfrac{tu}{\pi})dudt.
  \end{equation}
  To do so we
  expand $W(n,t)$ by means of \eqref{eq:W_r-definition},
  move the $s$-integration to $\Re(s)=\e>0$, interchange integration and arrive at a constant times
  \begin{equation}
    \int_{(\e)}P_r(s)L(1+2s,\chi_d)\frac{e^{s^2}}{s}
    \left(\frac{\abs{d}}{4\pi^2n}\right)^s
    \int_{-T^\e}^{T^\e}\cos(x\cosh(u))g_s(u)duds
  \end{equation}
  where
  \begin{equation}\label{eq:gdef}
    g_s(u)=\int_{0}^\infty f_1(t/X)t^{1-r}h(t)(\tfrac{1}{4}+t^2)^s e(\tfrac{tu}{\pi})dt.
  \end{equation}
  Note that $g_s(u)$ is the Fourier transform at $-u/\pi$ of $$\psi_s(t)=1_{t>0}f_1(t/X)t^{1-r}h(t)(\tfrac{1}{4}+t^2)^s.$$

  Doing $l$ successive partial integrations, we find, assuming  $1-r+2\Re(s)+6M\geq 0$, that
  \begin{equation}\label{eq:g-bound} g_s(u)\ll_{l}\frac{\norm{\psi_s^{(l)}}_{L^1(\R)}}{\abs{u}^l}\ll_{M,l}  (1+\abs{s})^l\frac{T^{2-r+2\Re(s)}}{X^l\abs{u}^l}.
  \end{equation}
  This allows us to estimate the contribution of the $u$ integral coming from $T^{-1+\beta}\leq \abs{u}$ (here we want $\beta$ to be small) as
  \begin{equation}
    \int_{(\e)}P_r(s)L(1+2s,\chi_d)\frac{e^{s^2}}{s}
    \left(\frac{\abs{d}}{4\pi^2n}\right)^s
    \int_{T^{-1+\beta}\leq \abs{u}\leq T^\e}\hspace{-1.5cm}\cos(x\cosh(u))g_s(u)duds\ll \frac{T^{2-r+2\e}}{X^lT^{(-1+\beta)(l-1)}}
  \end{equation}
  where the implied constant depends on $d,M,r,\e,l$, and we have assumed that $M$ is large enough such that \eqref{eq:g-bound} holds. Recall that $X=T^\alpha$. If $\alpha+\beta>1$ then the above error is less than $T^{1+\beta+ 2\e-(\alpha+\beta-1)l}$ which can be made small by choosing $l$ large.

  It remains to estimate
  \begin{equation}
    \int_{(\e)}P_r(s)L(1+2s,\chi_d)\frac{e^{s^2}}{s}
    \left(\frac{\abs{d}}{4\pi^2n}\right)^s
    \int_{\abs{u}\leq {T^{-1+\beta}}}\hspace{-1cm}\cos(x\cosh(u))g_s(u)duds.
  \end{equation}
  If we simply bound the inner integral trivially using \eqref{eq:g-bound} with $l=0$ this will lead to a total bound of $T^{1+\beta+\e}$. To get better estimates we focus on the $u$-integral and consider
  \begin{equation}\label{eq:oscillating-integral}
    F_\pm=\int_{\abs{u}\leq {T^{-1+\beta}}}\hspace{-1cm}e^{\pm i(x\cosh(u))}g_s(u)du.
  \end{equation}
  For $\abs{u}\leq T^{-1+\beta}$ the Taylor expansion of $\cosh(u)=1+u^2/2+ O(T^{4((-1+\beta))})$ gives that for $x\leq T^2$, and $\beta<1/2$ we have \begin{equation}e^{\pm i(x\cosh(u))}=e^{\pm ix}e^{\pm i\tfrac x2 u^2}+O \left(\frac{xT^{4\beta}}{T^4}\right). \end{equation}
  Using \eqref{eq:g-bound} with $l=0$, and the assumption on the upper bound of $x$  we see
  \begin{equation}
    F_\pm=G_\pm+O\left(x\frac{T^{4\beta}T^\beta}{T^4T}T^{2-r+2\Re(s)}\right)=G_\pm+O\left(\frac{T^{5\beta+2\Re(s)}}{T}\right),
  \end{equation}
  where
  \begin{equation}
    G_\pm=e^{\pm ix}\int_{\abs{u}\leq {T^{-1+\beta}}}e^{\pm \tfrac x2 i u^2}g_s(u)du.
  \end{equation}

  The strategy is now to use \eqref{eq:g-bound} to extend the range of the $u$-integral to $\R$ (at a negligible cost), insert the definition of $g_s(u)$, interchange the order of integration, use that for $x>0$ we have
  \begin{equation}\label{eq:fresnel-integral}
    \int_{\R}e^{\pm i\tfrac x 2 u^2}e^{2itu}du= e^{\pm i\tfrac \pi 4}\sqrt{\frac{2\pi}{x}}e^{\mp i\tfrac{2}{x}t^2}
  \end{equation}
  by \cite[3.322.3]{GradshteynRyzhik:2007}, followed by partial integration in the $t$ integral. There is a small problem in this argument;  Fubini-Tornelli is not directly applicable as \eqref{eq:fresnel-integral} does not converge absolutely. In order to circumvent this, we first use \eqref{eq:g-bound} to extend the integral to $\abs{u}\leq L$ for a large value $L$ to get
  \begin{equation}
    G_\pm=e^{\pm ix}\int_{\abs{u}\leq {L}}e^{\pm i\tfrac x2  u^2}g_s(u)du+O((1+\abs{s})^lT^{2+2\Re(s)-(\alpha+\beta-1)l}),
  \end{equation}
  where the error is small if $l$ is large. We now insert \eqref{eq:gdef} and interchange integration, which is now allowed. We get
  \begin{equation}
    G_\pm=e^{\pm ix}\int_0^\infty\psi_s(t)\int_{\abs{u}\leq L} e^{\pm i\tfrac x2  u^2}e^{2itu}du dt+O\left((1+\abs{s})^lT^{2+2\Re(s)-(\alpha+\beta-1)l}\right).
  \end{equation}
  The inner integral can be analyzed by integrating by parts twice, and we find that
  \begin{equation}
    \int_{\abs{u}\leq L}e^{\pm ixu^2}e^{2itu}du=\int_{\R}e^{\pm ixu^2}e^{2itu}du+O\left(\frac{1}{xL}\left(1+\frac{\abs{t}}{xL}+\frac{\abs{t}^2}{x}\right)\right).
  \end{equation}
  We use that in the $t$ integral the integrand is supported in $\abs{t}>X>1$, and we have $1/x\leq T^{100}$, so the error term is $O(\abs{t^2}T^{200}/L)$ in the range where it is applied. Combining the two preceding estimates and using \eqref{eq:fresnel-integral} we find that for $T^{-100}\leq x$
  \begin{equation}
    \abs{G_\pm}\leq \sqrt{\frac{2\pi}{x}}\abs{\int_0^\infty e^{\mp i\tfrac{2}{x}t^2} \psi_s(t)dt}+O((1+\abs{s})^l T^{2+2\Re(s)-(\alpha+\beta-1)l}) +O(\tfrac{T^{205}}{L}).
  \end{equation}
  Doing partial integration we find that as long as $1-r+6M +2\Re(s)\geq0$
  \begin{equation}
    \abs{\int_0^\infty e^{\mp i\tfrac{2}{x}t^2} \psi_s(t)dt }\ll \abs{x^l}\norm{\left(\frac{d}{dt}\frac{1}{t}\right)^l\psi_s}_{L^1(\R)}\ll x^lT^{2-r-2l\alpha+2\Re(s)}(1+\abs{s})^l.
  \end{equation}
  Since $x\leq T^{2-\e_0}$ we may choose $\alpha$ close enough to 1 that $2-\e_0<2\alpha$, after which this term will become negligible (take $\alpha=1-\beta/2$ and $0<\beta<\min(\varepsilon,\varepsilon_0)$). Collecting the above bounds, (choosing $L$ large enough and noting that $\alpha+\beta>1$) proves the claim.
\end{proof}

\subsection{Proof of Theorem \ref{thm:Twisted-first-moment}}
We now have the tools to prove Theorem \ref{thm:Twisted-first-moment}.
We start from the discrete spectrum contribution and apply the approximate functional equation \eqref{eq:approximate functional equation}
\begin{equation}
  \EuScript{D}
  =
  \sum_{t_j} \frac{L(\tfrac 12,\phi_j\times  f_\ck)\lambda_j(\ell)}{L(1,\sym^2\phi_j)}h(t_j)
  =
  2\sum_{n\geq 1} \frac{\lambda_{\ck}(n)}{\sqrt{n}} \sum_{t_j} \frac{\lambda_j(n)\lambda_j(\ell)}{L(1,\sym^2\phi_j)}h_n(t_j).
\end{equation}
As we saw in the proof of Lemma \ref{0112:lemma} (see \eqref{eq:h_n-bound}), the test function $h_n(t)$
decays fast enough both in $n$ and $t$ to ensure absolute convergence of the double sum.
Moreover, $h_n$ is even and holomorphic in the strip $|\Im(t)|\leq A$ for any $A>0$, since $W(n,t)$ and $h(t)$ are.
Moreover $h(t)$ decays exponentially in any fixed strip $\abs{\Im(t)}\leq 1/2+\delta$ and $W(n,t)$ grows at most polynomially in the same strip by applying Stirling to $\gamma(s,t)$ and then bounding the integral. Therefore the same-sign Kuznetsov formula \cite[Theorem 16.3]{IwaniecKowalski:2004}, \cite[Theorem 2.1]{HumphriesKhan:2023} applies.

The diagonal term gives
\begin{equation}\label{1402:eq001}
  \frac{\lambda_{\ck}(\ell)}{\sqrt{\ell}} \int_{-\infty}^{\infty} th(t)W(\ell,t) \tanh(\pi t)\frac{dt}{\pi^2}.
\end{equation}
By the first bound in Lemma \ref{0112:lemma} with $\eta=1/2-\e$ we can replace $W(\ell,t)$ by $L(1,\chi_d)$
up to an error. Combining this with the bound \eqref{eq:basics of lambda_chi} for $\lambda_{\ck}(\ell)$, we obtain that
\eqref{1402:eq001} equals, for any $\varepsilon>0$,
\begin{align}
  \frac{\lambda_{\ck}(\ell)}{\sqrt{\ell}} L(1,\chi_d) \int_{-\infty}^\infty th(t) \tanh(\pi t)\frac{dt}{\pi^2}
  +
  O_{d,\varepsilon}\biggl(|\lambda_{\ck}(\ell)|\int_{-\infty}^{\infty} |t^\varepsilon h(t)|dt\biggr)
  \\
  =
  \frac{\lambda_{\ck}(\ell)}{\ell^{\frac 12}} C_{M,d} T^2+ O_{M,d,\varepsilon}(T^{1+\varepsilon}\ell^{\varepsilon}).
\end{align}
The Kloosterman sum part, namely
\begin{equation}
  \sum_{n\geq 1}\frac{\lambda_\ck(n)}{\sqrt{n}}\sum_{c\geq 1} \frac{S(n,\ell,c)}{c}(\EuScript{K}^{+}h_n)\left(\frac{4\pi\sqrt{n\ell}}{c}\right),
\end{equation}
is bounded by applying Lemma \ref{2911:lemma1}, the bound \eqref{eq:basics of lambda_chi} on $\lambda_{\ck}(n)$
and Weil's bound  \begin{equation}\label{eq:Weil's bound}\abs{S(n,\ell,c)}\leq d(c)\sqrt{(n,\ell,c)c} \end{equation} on the Kloosterman sum.
The bound \eqref{2911:eq002} with $A=1/4+\e$ and $B$ sufficiently large shows that
the contribution from $n\geq T^{2+\e}$ is $O(\ell^{1/4+\e} T^{3/2-\e B+\e})$ which is negligibly small if $B$ is large. Therefore, we reduce to bounding the contribution from $n\leq T^{2+\e}$.
Next we  use \eqref{2911:eq003} when $c$ is small, and \eqref{2911:eq002} with $A=B>\frac{1}{4}+\e$ when $c$ is large,
obtaining
\begin{multline}
  \ll
  T^{-1+\e} \sum_{n\leq T^{2+\e}}\sum_{c\leq C} \frac{(nc)^\e \sqrt{(n,\ell,c)}}{\sqrt{nc}}
  +
  T^2\ell^{A} \sum_{n\leq T^{2+\e}}\sum_{c> C} \frac{(nc)^\e \sqrt{(n,\ell,c)}}{\sqrt{n}c^{2A+1/2}}
  \\
  \ll
  \left(C^{1/2} + \frac{T^{3}\ell^A}{C^{2A-1/2}}\right)(CT\ell)^\e
  \ll
  T^{\frac{3}{4A}+\e}\ell^{\frac{1}{4}+\e},
\end{multline}
when choosing $C=T^{\tfrac{3}{2A}}\ell^{1/2}$. We note that with this choice of $C$ the conditions for applying \eqref{2911:eq003} are satisfied. Choosing $A$ large enough gives a bound of $\ell^{1/4+\e}T^{\e}$.

Finally, the continuous part reads
\begin{equation}
  \EuScript{C} = \frac{1}{2\pi} \int_{-\infty}^{+\infty} \frac{1}{2}\biggl(2\sum_{n\geq 1} \frac{\lambda_{\ck}(n)}{\sqrt{n}}W(n,t) \lambda_t(n)\biggr) \frac{\lambda_t(\ell)h(t)}{|\zeta(1+2it)|^2}dt.
\end{equation}
By the approximate functional equation \eqref{eq:approximate functional equation}, the factorization \eqref{eq:factorisation when Eisenstein}, and the Euler product \eqref{eq:euler-product-class-group-L} we see that the term inside the big parenthesis equals \begin{equation}
  \abs{L(1/2+it, f_\ck)}^2+R ,\end{equation} where we have used \eqref{eq:basics of lambda_chi} to conclude that $\overline{L(1/2+it,f_\ck)}=L(1/2-it,f_\ck)$. We note that the
Stirling asymptotics
and standard bounds on Dirichlet $L$-functions at the edge of the critical strip and using \eqref{eq:Rbd} gives that $\abs{\zeta(1+2it)}^{-2}R\ll(1+\abs{t})^{\e -\tfrac{1}{2}}e^{-t^2}$ and therefore the integral over $t$ of the terms involving $R$ contributes $O(1)$.

Finally, using $|\zeta(1+2it)|^{-1}\ll (1+|t|)^\varepsilon$ along with
second moment estimates on Hecke $L$-functions (see
\cite[Proposition 1]{BlomerHarcos:2008a}
),
the continuous part contributes $O(T^{1+\varepsilon}\ell^\varepsilon)$ which finishes the proof of Theorem \ref{thm:Twisted-first-moment}.
\qed

\section{Fractional moment} \label{sectionfractional}

\subsection{The mollifier and key inequalities} We now construct mollifiers for central $L$-values of Rankin--Selberg $L$-functions $L(s,\phi_j \times f_{\ck})$ that behave like Euler products and use them to give pointwise bounds for the product of fractional central $L$-values, building upon the method of Radziwi{\l}{\l} and Soundararajan \cite{RadziwillSoundararajan:2015a}. Recalling the Hecke relations \eqref{eq:Hecke-relations} we find that the Hecke eigenvalues $\lambda_j(n)$ of $\phi_j$ satisfy
\begin{equation}
  \lambda_j(p)^a=\sum_{c=0}^a h_a(c) \lambda_j(p^c),
\end{equation}
for certain non-negative integers $h_a(c)$. It is well-known (see \cite[Lemma 3]{ConreyDukeFarmer:1997}) that
\begin{equation}
  h_a(c)=\frac{2}{\pi} \int_0^{\pi} (2 \, \cos \theta)^a \sin((c+1)\theta) \, \sin \theta \, d\theta,
\end{equation} from which we get that
\begin{equation} \label{eq:h1}
  0 \le h_a(c) \le 2^{a+1}.
\end{equation}
Additionally, it is easy to see that
\begin{equation} \label{eq:h2}
  h_a(0)=\begin{cases}
    1 & \text{ if } \,  a=2,       \\
    0 & \text{ if } \,  2 \nmid a.
  \end{cases}
\end{equation}
Let $\nu$ be the multiplicative function
with $\nu(p^a)=1/a!$ and note
that for any $k \in \mathbb N$ we have
\begin{equation}
  \sum_{p_1 \cdots p_k=n} 1=k! \nu(n).
\end{equation}
Given $\{ b(p) \}_p \subset \mathbb R$ with $\abs{b(p)}\le 4$, an interval $I\subset [2, \infty)$ and a Hecke--Maass form $\phi_j$ we define
\begin{equation} \label{eq:pdef}
  P_I(b,j)=\sum_{p \in I} \frac{b(p)\lambda_j(p)}{\sqrt{p}}.
\end{equation}
We note that $P_I(b,j)$ is real.
Further, we write $b(n)=\prod_{p^a || n} b(p)^a$. We have for any integer $k \ge 1$
\begin{equation} \label{eq:pexpand}
  P_I(b,j)^k=k! \sum_{\substack{p |n \Rightarrow p \in I \\ \Omega(n)=k}} \frac{b(n) \nu(n)}{\sqrt{n}} \ell_j(n)
\end{equation}
where $\ell_j$ is the multiplicative function with
\begin{equation} \label{eq:elldef}
  \ell_j(p^a)=\lambda_j(p)^a=\sum_{c=0}^a h_a(c) \lambda_j(p^c).
\end{equation}

For each $\ck \in \widehat{\Cl}_K$ we define the multiplicative function $g_{\ck}$ given by
\begin{equation} \label{eq:fdef}
  g_{\ck}(p^a)=\sum_{c=0}^a \frac{h_a(c)}{p^{c/2}} \lambda_{\ck}(p^c).
\end{equation}
Recall from \eqref{eq:basics of lambda_chi} that $\lambda_\ck(p^c)$ is real and satisfies $\abs{\lambda_\ck(p^c)}\leq d(p^c)$.
Using \eqref{eq:h1} and \eqref{eq:h2} we see for each $\ck \in \widehat{\Cl}_K$ we have the simple bound
\begin{equation} \label{eq:fbd}
  |g_{\ck}(p^a)|\le \begin{cases}
    \frac{2^{4a}}{\sqrt{p}} & \text{ if } 2 \nmid a, \\
    2^{4a}                  & \text{ if } 2|a.
  \end{cases}
\end{equation}
Given $\ck \in \widehat{\Cl}_K$ we define the completely multiplicative function $b_{\ck}$, which is given by
\begin{equation} \label{eq:bdef}
  b_{\ck}(p)=
  -\frac12 \lambda_{\ck}(p).
\end{equation}
We note that $\abs{b_\ck(p)}\leq 1$. Define also the completely multiplicative function \begin{equation} \label{eq:b12def}
  b_{\ck,\ck'}(p)=b_{\ck}(p)-b_{\ck'}(p)=\frac12(\lambda_{\ck'}(p)-\lambda_{\ck}(p)), \end{equation}
which we note satisfies $\abs{b_{\ck,\ck'}(p)}\leq 2$.

For each integer $r \ge 1$ define for $T\geq e^{e^{C_1}}$
\begin{equation} \label{eq:Ndef}
  N_1=2\lceil C_1 \log \log T \rceil \quad \text{ and } \quad N_{r+1}=2 \lceil C_1 \log N_r \rceil.
\end{equation}
Here  $C_1>100$ is an absolute constant that is sufficiently large in terms of other absolute constants.
Let $R$ be the largest positive integer such that $N_R>C_1^2$. Clearly,
\begin{equation} \label{eq:NRbd1}
  N_R \le e^{C_1}.
\end{equation}
Also, note that since
\begin{equation} \notag
  C_1^2< N_R= 2 \lceil C_1 \log N_{R-1} \rceil \le 2 (C_1 \log N_{R-1}+1)
\end{equation}
we have
\begin{equation}
  N_{R-1}> e^{C_1/2-1/C_1}> C_1^4.
\end{equation}
Iterating this argument, we obtain the crude bounds $N_{R-j+1} > C_1^{2j}$ for each $j=1,\ldots, R$. Since in particular $N_1>C_1^{2R}$ we verify that such a maximal $R$ exists, and we obtain the crude, yet sufficient bounds
\begin{equation} \label{eq:Rbd2}
  R=O(\log \log \log T)
\end{equation}
as well as
\begin{equation} \label{eq:recipbd}
  \sum_{r=1}^R \frac{1}{N_r} < \sum_{j=1}^{\infty}
  \frac{1}{C_1^{2j}} <C_1^{-1}.
\end{equation}
Also, define $x_r=T^{1/N_r^2}$ for each $r=1,\ldots, R$ and
\begin{equation} \label{eq:Idef}
  I_1=(2^{14}, x_1] \quad \text{ and } \quad I_r=(x_{r-1}, x_r], \quad r=2,\ldots,R.
\end{equation}

We also let $\delta_{\ck}=1/2$ if $\ck$ is a genus character and $\delta_{\ck}=1/4$ otherwise. For each $r=1,\ldots, R$ we define
\begin{align} \label{eq:molldef}
  M_r(j,\ck) & =(\log x_r)^{\delta_{\ck}} \exp(P_{(2^{14},x_r]}(b_{\xi},j)).
\end{align}
To motivate this definition, we note that
by comparing with \eqref{eq:dominant-term-Euler-prod} we see that $P_{(2^{14}, x_r]}(b_{\xi},j)$ equals $-1/2$ times the sum of the  $2^{14}<p\leq x_r$, $n=1$ terms of the Dirichlet representation of  $\log L(s,\phi_j\times f_\ck)$, which are then evaluated at $s=1/2$.

We note also that the sum of the $2^{14}<p\leq x_r$, $n=2$ terms of the Dirichlet representation of  $\log L(s,\phi_j\times f_\ck)$, evaluated at $s=1/2$ (compare \eqref{eq:secondary-term-Euler-prod}), and then multiplied by  $-1/2$ equals
\begin{equation}\label{eq:contribution from secondary terms}
  -\frac{1}{2}\sum_{2^{14}<p\leq x_r}\frac{(\lambda_j(p^2)-1)(\lambda_\ck(p^2)+\chi_d(p)-2)}{2p}.
\end{equation}
By the ideal theory for quadratic imaginary fields \cite[page 57]{IwaniecKowalski:2004} a small calculation shows that the second factor in the numerator equals $\lambda_{\ck^2}(p)+1-\chi_d(p)$ \emph{except} at the finitely many primes dividing $d$. Using that $\lambda_j(p^2)$ is the $p$th coefficient of the symmetric square of $\phi_j$, and that for a genus character \eqref{eq:Fourier-coefficients when genus} $\lambda_{\ck_{d_1,d_2}}(p)=\chi_{d_1}(p)+\chi_{d_2}(p)$  Selberg's orthogonality conjecture \cite{LiuWangYe:2005, LiuYe:2005, LiuYe:2006, AvdispahicSmajlovic:2010} gives that \eqref{eq:contribution from secondary terms} has asymptotics
\begin{equation}
  \frac{1}{4}(1_{\ck\in \mathcal G_K}+1)\log\log x_r +O(1)=\delta_\ck\log\log x_r+O(1)
\end{equation}
where $1_{\ck\in \mathcal G_K}$ is the indicator function of $\ck$ being a genus character. Exponentiating the asymptotics gives the factor of $(\log x)^{\delta_\ck}$ in \eqref{eq:molldef}.

The function $M_r(j,\ck)$ therefore \lq imitates\rq{} the reciprocal of the root of the central $L$-value
\begin{equation}
  L(\tfrac{1}{2},\phi_j\times f_\ck)^{-\tfrac{1}{2}}=\exp(-\tfrac{1}{2}\log L(\tfrac{1}{2},\phi_j\times f_\ck)),
\end{equation}
and this approximation is more accurate for larger $r$ at the cost of becoming more difficult to handle.

Define also for $r=1,\ldots, R$
\begin{align} \label{eq:molldef2}
  \mathcal M_r(j,\ck,\ck') & =(\log x_r)^{\delta_{\ck}-\delta_{\ck'}} \prod_{q=1}^r \sum_{k=0}^{N_q} \frac{P_{I_q}(b_{\ck,\ck'},j)^k}{k!} \\
                           & =(\log x_r)^{\delta_{\ck}-\delta_{\ck'}} \prod_{q=1}^r \sum_{\substack{p|n \Rightarrow p \in I_q             \\ \Omega(n) \le N_q}} \frac{b_{\ck,\ck'}(n) \nu(n)}{\sqrt{n}} \ell_j(n),
\end{align}
where we have used \eqref{eq:pexpand}. In particular we have that $\mathcal M_r(j,\ck,\ck')$ is supported on integers $n \le T^{1/N_1+\cdots +1/N_r}< T^{1/C_1}$ by \eqref{eq:recipbd}.
For real $t \le N/e^2$ with $N\ge 2$ even \cite[Lemma 1]{RadziwillSoundararajan:2015a} implies
\begin{equation} \label{eq:exptaylor}
  e^t \le \bigg(1+\frac{e^{-N}}{16} \bigg) \sum_{k=0}^N \frac{t^k}{k!}.
\end{equation}
Hence, for $j$ such that $|P_{I_q}(b_{\ck},j)|,|P_{I_q}(b_{\ck'},j)|\le N_q/(2e^2)$ for each $q=1,\ldots, r$ we have that
\begin{equation} \label{eq:eulerprod}
  M_r(j,\ck)M_r(j,\ck')^{-1}=(\log x_r)^{\delta_{\ck}-\delta_{\ck'}} \prod_{q=1}^r \exp(P_{I_r}(b_{\ck,\ck'},j))
  \le C(r) \mathcal M_r(j,\ck,\ck'),
\end{equation}
where
\begin{equation} \label{eq:Crdef}
  C(r) =\prod_{q=1}^r \bigg(1+\frac{e^{-N_q}}{16} \bigg) \ll 1.
\end{equation}

With our mollifiers in hand, we will now derive the key inequalities, which bound 
$L(s,\phi_j \times f_{\ck})^{1/2}L(s,\phi_j \times f_{\ck'})^{1/2}$ 
in terms of the central $L$-values multiplied by short Dirichlet polynomials. For each $r=1,\ldots, R$ consider the sets
\begin{equation}
  A_r=\{ j :  |P_{I_q}(b,j)|\le N_q/(2e^2)\text{ for both } b=b_{\ck},b_{\ck'} \text{ and all } q=1,\ldots,r \}
\end{equation}
and
\begin{equation}
  B_r=\{j : |P_{I_r}(b,j)| \ge N_r/(2e^2) \text{ for } b=b_{\ck} \text{ or } b=b_{\ck'}\}
\end{equation}
where $b_{\ck}$ is as given in \eqref{eq:bdef}.   Let $S_0=B_1$ and $S_{R}=A_R$. We will now define sets $S_r$ for each $r=1,\dots, R-1$
so that $S_r$ consists of $j$'s where $r$ is the smallest integer such that $P_{I_{r+1}}(b,j)$ is unusually large for $b=b_{\ck}$ or $b=b_{\ck'}$. Concretely, for each $r=1,\ldots, R-1$ we
let
\begin{equation} \label{eq:Srsetdef}
  S_r=A_r \cap B_{r+1}.
\end{equation} We note that $S_0\cup \cdots \cup S_R$ is a disjoint union of the set of all $j$'s.  To shorten notation, we write
\begin{equation}\label{eq:script-L-def}
  \mathcal L(j,\ck)=\frac{L(\tfrac12, \phi_j \times f_{\ck})}{L(1,\tmop{sym}^2 \phi_j)}.
\end{equation}

\begin{lemma} \label{lem:keyineq} 
Let $\ck \in \widehat{\Cl}_K, \ck'\in\widehat{\Cl}_{K'}$ and let $\phi_j$ be a Hecke--Maass cusp form. Each of the following hold.
  \begin{enumerate}
    \item For $j \in S_R$ we have
          \begin{equation}
            2 \sqrt{\mathcal L(j,\ck) \mathcal L(j,\ck')} \le C(R) \left( \mathcal L(j,\ck) \mathcal M_R(j,\ck,\ck')+\mathcal L(j,\ck') \mathcal M_R(j,\ck',\ck)\right).
          \end{equation}
    \item For $j \in S_0$ we have
          \begin{equation}
            2 \sqrt{\mathcal L(j,\ck) \mathcal L(j,\ck')} \le \left(\mathcal L(j,\ck)+\mathcal L(j,\ck')\right) \left( \left( \frac{P_{I_{1}}(b_{\ck},j) 2e^2}{N_{1}}\right)^{N_{1}}+\left(\frac{P_{I_{1}}(b_{\ck'},j) 2e^2}{N_{1}} \right)^{N_{1}}  \right).
          \end{equation}
    \item  Suppose $r \in \{1, \ldots, R-1\}$. For $j \in S_r$ we have
          \begin{equation}
            \begin{split}
              2 \sqrt{\mathcal L(j,\ck) \mathcal L(j,\ck')} \le & C(r)\left( \mathcal L(j,\ck)\mathcal M_r(j,\ck,\ck') +\mathcal L(j,\ck')\mathcal M_r(j,\ck',\ck) \right)                                                         \\
                                                                & \times \left( \left( \frac{P_{I_{r+1}}(b_{\ck},j) 2e^2}{N_{r+1}}\right)^{N_{r+1}}+\left(\frac{P_{I_{r+1}}(b_{\ck'},j) 2e^2}{N_{r+1}} \right)^{N_{r+1}}  \right).
            \end{split}
          \end{equation}
  \end{enumerate}
\end{lemma}
\begin{proof}
  For each $r=1,\ldots, R$ and each $j \in S_r$ we have by \eqref{eq:eulerprod}
  \begin{equation} \label{eq:mbd}
    M_r(j,\ck)M_r(j,\ck')^{-1} \le C(r) \mathcal M_r(j,\ck,\ck') \text{
      and   } M_r(j,\ck')M_r(j,\ck)^{-1} \le C(r) \mathcal M_r(j,\ck',\ck).
  \end{equation}
  Additionally, for each $r=0,\ldots, R-1$ and $j \in S_r$
  \begin{equation} \label{eq:markov}
    1 \le \left( \frac{P_{I_{r+1}}(b_{\ck},j) 2e^2}{N_{r+1}}\right)^{N_{r+1}}+\left(\frac{P_{I_{r+1}}(b_{\ck'},j) 2e^2}{N_{r+1}} \right)^{N_{r+1}}.
  \end{equation}

  Apply \eqref{eq:elementary} with $L=\mathcal L(j,\ck)$ and $L'=\mathcal L(j,\ck')$. If $j \in S_R$ take $M=M_R(j,\ck)$, $M'=M_R(j,\ck')$ and then use \eqref{eq:mbd} with $r=R$ to get the first claim.
  In the case where $j \in S_0$ choose $M=M'=1$ and then use \eqref{eq:markov} with $r=0$ to obtain the second claim. Finally, if $j \in S_r$ with $r=1,\ldots,R-1$ take $M=M_r(j,\ck)$, $M'=M_r(j,\ck')$ and then use \eqref{eq:mbd} and \eqref{eq:markov} to get the third claim.
\end{proof}

\subsection{Preliminary Lemmas}

Using the formula for the twisted moment given in Theorem~\ref{thm:Twisted-first-moment}
we establish a follow-up version where in place of twisting by a single eigenvalue
we allow for a Dirichlet polynomial. This will later be used
to bound a mollified first moment and to control the number of exceptional $\phi_j$ for which 
$P_{I_r}(b_{\ck},j)$ or $P_{I_r}(b_{\ck'},j)$ 
is unusually large, in which case \eqref{eq:eulerprod} is not valid.

\begin{lem} \label{lem:expand}
  Let $\varepsilon>0$, $1 \le r \le R$, and $\ck \in \widehat{\Cl}_K$. Let $k_1,\ldots,k_r$ be nonnegative integers such that $\prod_{q=1}^r x_q^{k_q} \le T^{2-\varepsilon_0}$ for some $\varepsilon_0>0$.
  Then for $h$ as in \eqref{def:h} and $C_{M,d}$ as in Theorem \ref{thm:Twisted-first-moment}
  we have that
  \begin{equation}
    \begin{split}
               & \sum_{t_j >0} \mathcal L(j,\ck) \left(\prod_{q=1}^r  P_{I_q}(b,j)^{k_q} \right)
      h(t_j)                                                                                    \\
      \qquad ={} & C_{M,d}T^2 \left( \prod_{q=1}^r k_q!\sum_{\substack{p|n \Rightarrow p \in I_q  \\ \Omega(n) =k_q}} \frac{b(n) \nu(n)}{\sqrt{n}}g_{\ck}(n) \right) +O\left( T^{1+\varepsilon} \prod_{q=1}^r (2^{12} x_q)^{k_q/2}\right).
    \end{split}
  \end{equation}
  Here $g_{\ck}$ is as given in \eqref{eq:fdef}, $P_I(b,j)$ is as defined in \eqref{eq:pdef}, $I_r$ is as in \eqref{eq:Idef}, and $\mathcal L(j,\ck)$ as in \eqref{eq:script-L-def}.
\end{lem}
\begin{proof}
  Using \eqref{eq:pexpand} and multiplicativity we get that
  \begin{equation} \label{eq:applyexpand}
    \begin{split}
      \sum_{j \ge 1} \mathcal L(j,\ck)  \bigg(\prod_{q=1}^r  P_{I_q}(b,j)^{k_q} \bigg)
      h(t_j)
      ={} & k_1!\cdots k_r! \sum_{\substack{p|n_1 \Rightarrow p \in I_1     \\ \Omega(n_1)=k_1}} \cdots \sum_{\substack{p|n_r \Rightarrow p \in I_r\\ \Omega(n_r)=k_r}} \frac{b(n_1\cdots n_r) \nu(n_1\cdots n_r)}{\sqrt{n_1\cdots n_r}} \\
        & \times \sum_{j \ge 1} \mathcal L(j,\ck) \ell_{j}(n_1\cdots n_r)
      h(t_j).
    \end{split}
  \end{equation}
  Applying Theorem \ref{thm:Twisted-first-moment} and noting $n=n_1\cdots n_r \le \prod_{q=1}^r x_q^{k_q} \le T^{2-\e_0}$ we have
  \begin{equation} \label{eq:applymoment}
    \sum_{j \ge 1} \mathcal L(j,\ck) \ell_j(n)
    h(t_j)=C_{M,d}T^2 g_{\ck}(n)+O( 8^{\Omega(n)}T^{1+\e} ),
  \end{equation}
  where to estimate the error term we used the bound
  \begin{equation}
    \prod_{p^a || n } \sum_{c=0}^a h_a(c) \le \prod_{p^a||n} (a+1)2^{a+1} \le 8^{\Omega(n)}.
  \end{equation}
  Hence, combining \eqref{eq:applyexpand} and \eqref{eq:applymoment} we get, using $\abs{b(n)}\leq 4^{\Omega(n)}$, that
  \begin{equation} \label{eq:intermediate}
    \begin{split}
      \sum_{j \ge 1} \mathcal L(j,\ck) \bigg(\prod_{q=1}^r  P_{I_q}(b,j)^{k_q} \bigg)
      h(t_j)
      = {} & C_{M,d}T^2 \bigg( \prod_{q=1}^r k_q!\sum_{\substack{p|n \Rightarrow p \in I_q        \\ \Omega(n) =k_q}} \frac{b(n) \nu(n)}{\sqrt{n}}g_{\ck}(n) \bigg)\\
           & +O\bigg( T^{1+\e} \bigg( \prod_{q=1}^r k_q!\sum_{\substack{p|n \Rightarrow p \in I_q \\ \Omega(n) =k_q}} \frac{32^{\Omega(n)} \nu(n)}{\sqrt{n}} \bigg) \bigg).
    \end{split}
  \end{equation}
  We note that for each $q=1, \ldots, r$
  \begin{equation}
    k_q! \sum_{\substack{p|n \Rightarrow p \in I_q \\ \Omega(n)=k_q}} \frac{ 32^{\Omega(n)}\nu(n)}{\sqrt{n}}=\bigg( 32\sum_{p \in I_q}\frac{1}{\sqrt{p}} \bigg)^{k_q} \le (64x_q^{1/2})^{k_q}.
  \end{equation}
  Applying this bound in \eqref{eq:intermediate} completes the proof.\end{proof}

The next two lemmas will be used to estimate the sums appearing in the main term in Lemma \ref{lem:expand}.

\begin{lem} \label{lem:bound}
  Let $k$ be a nonnegative integer and $I \subset [2,\infty)$ be an interval. We have that
  \begin{equation}
    (2k)! \bigg| \sum_{\substack{p|n \Rightarrow p \in I \\ \Omega(n) =2k}}
    \frac{b(n) \nu(n)g_{\ck}(n)}{\sqrt{n}}\bigg| \le e^{O(k)}\bigg( \bigg(\sum_{p \in I} \frac1p\bigg)^{2k}+\bigg(k \sum_{p \in I} \frac1p\bigg)^{k}\bigg),
  \end{equation}
  where the implied constant is absolute.
\end{lem}
\begin{proof}
  Writing $n=u^2v$ where $u,v$ are integers and $v$ is squarefree then using \eqref{eq:fbd} we have that
  \begin{equation}
    \abs{g_\ck(u^2v)} \le \frac{2^{4 \Omega(u^2v)}}{\sqrt{v}}.
  \end{equation}
  Also, recall $|b(n)| \le 4^{\Omega(n)}$. Consequently, we have that
  \begin{equation} \label{eq:firstbd}
    \begin{split}
      \bigg| \sum_{\substack{p|n \Rightarrow p \in I                                        \\ \Omega(n)=2k}} \frac{
      b(n)\nu(n) g_{\ck}(n)}{\sqrt{n}}\bigg| & \le \sum_{\substack{p|uv \Rightarrow p \in I \\ \Omega(u^2v)=2k}} \frac{
      2^{6\Omega(u^2v)}\nu(u^2v)}{uv}                                                       \\
                                             & \le \sum_{\substack{p|uv \Rightarrow p \in I \\ \Omega(u^2v)=2k}} \frac{
        2^{12\Omega(uv)}\nu(u)\nu(v)}{uv},
    \end{split}
  \end{equation}
  where in the last step we used that $\nu(ab)\le \nu(a)\nu(b)$ and $\nu(u^2)\le \nu(u)$. We see that
  \begin{equation} \label{eq:secondbd}
    \begin{split}
      \sum_{\substack{p|uv \Rightarrow p \in I                                                                                  \\ \Omega(u^2v)=2k}} \frac{
      2^{12\Omega(uv)}\nu(u)\nu(v)}{uv}={} & \sum_{l=0}^k
      \sum_{\substack{p|u \Rightarrow p \in I                                                                                   \\ \Omega(u)=l}} \frac{
      2^{12\Omega(u)}\nu(u)}{u}\sum_{\substack{p|v \Rightarrow p \in I                                                          \\ \Omega(v)=2k-2l}} \frac{
      2^{12\Omega(v)}\nu(v)}{v}                                                                                                 \\
      ={}                                  & \sum_{l=0}^k \frac{\Big( 2^{12} \sum_{p \in I} \frac{1}{p}\Big)^{2k-l}}{l!(2k-2l)!}.
    \end{split}
  \end{equation}
  Observe that if $l\leq k$ then
  \begin{equation}
    \frac{(2k)!(k-l)!}{(2k-2l)!k!}=\frac{2k \cdot (2k-1)\cdots (2k-l+1)}{k\cdot (k-1)\cdots (k-l+1)}  (2k-l)\cdots (2k-2l+1) \le (4k)^l,
  \end{equation}
  where we have used that $\frac{2k-i}{k-i}(2k-l-i)\leq 4k$ when $0\leq i\leq l-1.$
  Hence, using the preceding inequality we obtain
  \begin{equation} \label{eq:thirdbd}
    \begin{split}
      \sum_{l=0}^k \frac{\Big( 2^{12} \sum_{p \in I} \frac{1}{p}\Big)^{2k-l}}{l!(2k-2l)!}
      \le{} & \frac{\Big( 2^{12} \sum_{p \in I} \frac{1}{p}\Big)^{k}}{(2k)!}
      \sum_{l=0}^k \binom{k}{l} (4k)^l \bigg(2^{12} \sum_{p \in I} \frac{1}{p} \bigg)^{k-l} \\
      ={}   & \frac{\Big( 2^{12} \sum_{p \in I} \frac{1}{p}\Big)^{k}}{(2k)!}
      \bigg(2^{12} \sum_{p \in I} \frac1p+4k \bigg)^k.
    \end{split}
  \end{equation}
  Combining \eqref{eq:firstbd}, \eqref{eq:secondbd}, and \eqref{eq:thirdbd} completes the proof.
\end{proof}

The subsequent lemma will be used to estimate the main term that arises from the mollified first moment of $L(\tfrac12,\phi_j\times f_{\ck})$.
\begin{lem} \label{lem:euler}
  Let $N\ge 1$ and $I \subset [2^{14}, \infty)$ be an interval. Then
  \begin{equation}
    \begin{split}
      \sum_{\substack{p|n \Rightarrow p \in I                                                                                                  \\ \Omega(n) \le N}} \frac{b(n) \nu(n) g_{\ck}(n)}{\sqrt{n}} ={} &\bigg(1+O\bigg(2^{-N} \exp\bigg(2^{16}\sum_{p\in I} \frac{1}{p}\bigg)\bigg) \bigg)\\
       & \times \prod_{p \in I} \bigg( 1+\frac{b(p)g_{\ck}(p)}{\sqrt{p}}+\frac{b(p)^2 g_{\ck}(p^2)}{2p}+O\bigg(\frac{1}{p^{3/2}} \bigg)\bigg).
    \end{split}
  \end{equation}
\end{lem}
\begin{proof}
  For $\Omega(n)>N$ we have that $2^{\Omega(n)-N} \ge 1$, hence
  \begin{equation}
      \sum_{\substack{p|n \Rightarrow p \in I                                                \\ \Omega(n) \le N}} \frac{b(n) \nu(n) g_{\ck}(n)}{\sqrt{n}}
      = \sum_{\substack{p|n \Rightarrow p \in I }} \frac{b(n) \nu(n) g_{\ck}(n)}{\sqrt{n}} 
         +O\bigg( 2^{-N} \sum_{p |n \Rightarrow p \in I}\frac{
      8^{\Omega(n)} \nu(n) |g_{\ck}(n)|}{\sqrt{n}}\bigg).
      \end{equation}
  Using multiplicativity we see that
  \begin{equation}
    \sum_{p |n \Rightarrow p \in I}\frac{
    8^{\Omega(n)} \nu(n) |g_{\ck}(n)|}{\sqrt{n}}= \prod_{p \in I}\bigg(\sum_{l=0}^\infty \frac{2^{3l}|g_\ck(p^l)|}{l!p^{l/2}} \bigg)
  \end{equation}
  Using \eqref{eq:fbd} and noting $p \in I$ implies $p\ge 2^{14}$, we may bound this by
  \begin{align}
    \prod_{\substack{ p \in I         \\ }}  \left(1+\frac{2^3\abs{g_\ck(p)}}{p^{1/2}}+\frac{1}{p}\sum_{l=2}^\infty\frac{2^{3l}2^{4l}}{l!p^{(l-2)/2}} \right)      
     & \le  \prod_{\substack{ p \in I \\ }} \left(1+\frac{2^7}{p}+\frac{1}{p}\sum_{l=2}^\infty\frac{2^{3l}2^{4l}}{l!2^{14(l-2)/2}} \right) \\
     & \le  \prod_{\substack{ p \in I \\ }} \left(1+\frac{2^{15}}{p} \right) \le \exp\bigg(2^{15} \sum_{p \in I} \frac{1}{p} \bigg),
  \end{align}
  where we have used that $\log (1+t) < t$ for $t>0$. By a similar argument that we will omit, only now using $\log(1-t)>-2t$
  for $0<t<1/2$, we also have
  \begin{equation} \label{eq:lowerbdeuler}
    \sum_{p|n \Rightarrow p \in I}
    \frac{b(n)\nu(n)g_{\xi}(n)}{\sqrt{n}} \ge
    \prod_{p \in I}\bigg(1-\sum_{l=1}^\infty \frac{2^{2l}|g_\ck(p^l)|}{l!p^{l/2}} \bigg) \ge
    \exp\bigg(-2^{15} \sum_{p \in I} \frac{1}{p} \bigg).
  \end{equation}
  Also, we have that
  \begin{equation}
    \sum_{\substack{p|n \Rightarrow p \in I }} \frac{b(n) \nu(n) g_{\ck}(n)}{\sqrt{n}}
    =\prod_{p \in I}\bigg(1+\frac{b(p)g_{\ck}(p)}{\sqrt{p}}
    +\frac{b(p)^2g_{\ck}(p^2)}{2p}+O\bigg( \frac{1}{p^{3/2}}\bigg) \bigg).
  \end{equation}
  Combining the preceding estimates completes the proof.
\end{proof}

\subsection{Proof of Theorem \ref{thm:fracmoments}}

We begin with the following lemma.

\begin{lem} \label{lem:orthog}
  Let $\ck \in \widehat{\Cl}_{\mathbb Q(\sqrt{d})}$, $\ck' \in \widehat{\Cl}_{\mathbb Q(\sqrt{d'})}$.
  Each of the following hold.
  \begin{enumerate}
    \item Suppose at least one of $\ck$ and $\ck'$ is not a genus character. Then
          \begin{equation}
            \sum_{p \le x}\frac{\lambda_{\ck}(p)\lambda_{\ck'}(p)}{p}=O(1).
          \end{equation}
    \item Suppose $\ck$, $\ck'$ are both genus characters corresponding to

          $d=d_1d_2$ and $d'=d_1'd_2'$. Let $B_{\ck,\ck'}=\#\{ d_j=d_i' : (i,j) \in \{1,2\}^2\}$. Then
          \begin{equation}
            \sum_{p \le x}\frac{\lambda_{\ck}(p)\lambda_{\ck'}(p)}{p}=B_{\ck,\ck'} \log \log x+O(1).
          \end{equation}
    \item Suppose $\ck$ is not a genus character. Then
          \begin{equation}
            \sum_{p \le x} \frac{\lambda_{\ck}(p)^2}{p}=\log \log x+O(1).
          \end{equation}
    \item Suppose $\ck$ is a genus character. Then
          \begin{equation}
            \sum_{p \le x} \frac{\lambda_{\ck}(p)^2}{p}=2 \log \log x+O(1).
          \end{equation}
  \end{enumerate}
\end{lem}
\begin{proof}
  The first and third claims follow from
  \cite[Corollary 1.5]{LiuWangYe:2005}, \cite[Corollary 1.5]{LiuYe:2005}.
  Even if our situation is self-contragredient, we note that this extends to the non-self-contragredient case by \cite{LiuYe:2006,AvdispahicSmajlovic:2010}.
  Using \eqref{eq:Fourier-coefficients when genus}, the second and fourth claims are easy consequences of a classical result of Dirichlet, i.e. for any non-principal Dirichlet character $\chi \Mod q$ we have
  \begin{equation}
    \sum_{p \le x} \frac{\chi(p)}{p}=O_q(1).\qedhere
  \end{equation}
\end{proof}

\begin{proof}[Proof of Theorem \ref{thm:fracmoments}] Recall the definitions of $N_r$ in \eqref{eq:Ndef}, $I_r$ in \eqref{eq:Idef}, and $S_r$ in \eqref{eq:Srsetdef}.   Let $\mathcal M_r(j,\ck,\ck')$ be as given in \eqref{eq:molldef2} and $\mathcal L(j,\ck)$ be as given in \eqref{eq:script-L-def}.

  Observe that $1_{(T,2T]}(t)\ll h(t)$ where $1_{(T,2T]}$ is the indicator function of the interval $(T,2T]$ and $h$ is given in \eqref{def:h}. Recalling that $S_0 \cup \cdots \cup S_R$ is a disjoint union of all $j$'s, we have
  \begin{equation} \label{eq:fracbd}
    \begin{split}
      \sum_{T < t_{j} \le 2T} \sqrt{\mathcal L(j, \ck)\mathcal L(j, \ck')} & \ll \sum_{r=0}^{R}\sum_{j \in S_r} \sqrt{\mathcal L(j, \ck)\mathcal L(j, \ck')} h(t_j).
    \end{split}
  \end{equation}

  We first consider $r=1,\ldots,R-1$. Using non-negativity of the central $L$-values we have $\mathcal L(j,\ck) \ge 0$. Also  $\mathcal M_r(j,\ck,\ck')>0$,
  which follows from \eqref{eq:eulerprod}. Hence, applying Lemma \ref{lem:keyineq} we have for each $r=1,\ldots, R-1$
  \begin{equation} \label{eq:mixedfirstbd}
    \begin{split}
      2 \sum_{j \in S_r} \sqrt{\mathcal L(j,\ck) \mathcal L(j,\ck')} h(t_j) \le{} & C(r) \sum_{t_j>0} \left( \mathcal L(j,\ck)\mathcal M_r(j,\ck,\ck') +\mathcal L(j,\ck')\mathcal M_r(j,\ck',\ck) \right)                                                 \\
                                                                                & \times \left( \left( \frac{P_{I_{r+1}}(b_{\ck},j) 2e^2}{N_{r+1}}\right)^{N_{r+1}}+\left(\frac{P_{I_{r+1}}(b_{\ck'},j) 2e^2}{N_{r+1}} \right)^{N_{r+1}}  \right) h(t_j)
    \end{split}
  \end{equation}
  where $C(r)$ is as given in \eqref{eq:Crdef} and satisfies $C(r) \ll 1$.
  Applying \eqref{eq:molldef2} and Lemma \ref{lem:expand}, which is valid by \eqref{eq:recipbd}, we see that
  
  \begin{align} \label{eq:momentest1}
      \sum_{t_j>0} \mathcal L(j,\ck) \mathcal M_r(j,\ck,\ck') & P_{I_{r+1}}(b,j)^{N_{r+1}} h(t_j)
 \\     &= C_{M,d}T^2(\log x_r)^{\delta_{\ck}-\delta_{\ck'}} \mathcal A_r(b,\ck,\ck')
      +O(T^{1+1/C_1+\varepsilon}),
 \end{align} 

  for $b=b_{\ck}$ or $b=b_{\ck'}$  where
  \begin{equation}
    \mathcal A_r(b,\ck,\ck')=N_{r+1}! \sum_{\substack{p |n \Rightarrow p \in I_{r+1} \\ \Omega(n)=N_{r+1}}} \frac{b(n)\nu(n) g_{\ck}(n)}{\sqrt{n}} \prod_{q=1}^r \sum_{\substack{p |n \Rightarrow p \in I_{q} \\ \Omega(n)\le N_{q}}} \frac{b_{\ck,\ck'}(n)\nu(n) g_{\ck}(n)}{\sqrt{n}}.
  \end{equation}
  To estimate the error term we used that $ \prod_{q=1}^{r+1}(2^{12}x_q)^{N_q/2} \le  T^{1/N_1 +\cdots +1/N_{r+1}} \le T^{1/C_1}$, by \eqref{eq:recipbd}, and $\prod_{q=1}^r\sum_{k=0}^{N_q} \frac{1}{k!}<e^r<T^{\varepsilon}$ by \eqref{eq:Rbd2}.

  We note that, $\sum_{p \in I_1} p^{-1} \sim \log \log T \sim N_1/(2C_1)$, and for each $r=2,\ldots, R$
  \begin{equation} \label{eq:mertens}
    \sum_{p \in I_r} \frac{1}{p}=\log \frac{\log x_r}{\log x_{r-1}}+O(1)=\frac{N_r}{C_1}+O(\log N_r).
  \end{equation}
  Hence, using \eqref{eq:mertens} and that $C_1$ is sufficiently large along with Lemma \ref{lem:euler} we have for $r=1,\ldots, R$ that
  \begin{equation} \label{eq:eulerbd2}
    \begin{split}
       & \prod_{q=1}^r \sum_{\substack{p |n \Rightarrow p \in I_{q}                                                                                                                                          \\ \Omega(n)\le N_{q}}} \frac{b_{\ck,\ck'}(n)\nu(n) g_{\ck}(n)}{\sqrt{n}} \\
       & \ll \prod_{q=1}^r  \bigg(1+O(e^{-N_q/2})\bigg) \prod_{p \in I_q} \bigg(1+\frac{b_{\ck,\ck'}(p)g_{\ck}(p)}{\sqrt{p}}+\frac{b_{\ck,\ck'}(p)^2g_{\ck}(p^2)}{2p} +O\left(\frac{1}{p^{3/2}}\right)\bigg) \\
       & \ll \prod_{p \le x_r} \bigg(1+\frac{b_{\ck,\ck'}(p)g_{\ck}(p)}{\sqrt{p}}+\frac{b_{\ck,\ck'}(p)^2g_{\ck}(p^2)}{2p} +O\left(\frac{1}{p^{3/2}}\right) \bigg).
    \end{split}
  \end{equation}
  By Lemma \ref{lem:bound} with $k=N_{r+1}/2$ and \eqref{eq:mertens} we have for each $r=0,\ldots, R-1$ and each $b=b_{\ck},b_{\ck'}$
  \begin{equation} \label{eq:markovsave}
    \begin{split}
      N_{r+1}! \sum_{\substack{p |n \Rightarrow p \in I_{r+1}                                                                      \\ \Omega(n)=N_{r+1}}} \frac{b(n)\nu(n) g_{\ck}(n)}{\sqrt{n}}
       & \le  e^{O(N_{r+1})}  \bigg( \frac{N_{r+1}^2}{2C_1} \bigg)^{N_{r+1}/2}\le \bigg( \frac{N_{r+1}}{C_1^{1/3}}\bigg)^{N_{r+1}}
    \end{split}
  \end{equation}
  since $C_1$ is sufficiently large.

  To bound the Euler product on the right-hand side of \eqref{eq:eulerbd2} observe that
  \begin{equation}
    g_{\ck}(p)=\frac{\lambda_{\ck}(p)}{\sqrt{p}} \textrm{ and } g_{\ck}(p^2)=1+\frac{\lambda_{\ck}(p^2)}{p}=1+O\bigg(\frac{1}{p}\bigg)\end{equation} as follows from \eqref{eq:h2} and \eqref{eq:fdef}. Using these formulas and recalling \eqref{eq:b12def} we see that
  \begin{equation} \label{eq:pnt}
    \begin{split}
      \sum_{p \le x_r} & \bigg( \frac{b_{\ck,\ck'}(p)g_{\ck}(p)}{\sqrt{p}}
      +\frac{b_{\ck,\ck'}(p)^2g_{\ck}(p^2)}{2p}+O(p^{-3/2})\bigg)                              \\
      =                & -\frac{3}{8}\sum_{p \le x_r} \frac{\lambda_{\ck}(p)^2}{p}+\frac{1}{8}
      \sum_{p \le x_r}\frac{\lambda_{\ck'}(p)^2}{p}+\frac14 \sum_{p \le x_r}
      \frac{\lambda_{\ck}(p) \lambda_{\ck'}(p)}{p}+ O(1).
    \end{split}
  \end{equation}

  Let $B_{\ck,\ck'}$ be as in Lemma \ref{lem:orthog} and define
  \begin{equation} \label{eq:etadef}
    \eta_{\ck,\ck'}=\begin{cases}
      1/4                       & \text{ if } \ck^2 \neq 1 \text{ and } (\ck')^2 \neq 1, \\
      1/8                       & \text{ if } \ck^2 \neq 1 \text{ and } (\ck')^2 = 1,    \\
      5/8                       & \text{ if } \ck^2 = 1 \text{ and } (\ck')^2 \neq 1,    \\
      1/2-\tfrac14 B_{\ck,\ck'} & \text{ if }  \ck^2=1 \text{ and } (\ck')^2=1.
    \end{cases}
  \end{equation}
  Applying Lemma \ref{lem:orthog} to estimate the sums on the right-hand side of \eqref{eq:pnt} we conclude for each $r=1,\ldots, R$ that
  \begin{equation} \label{eq:etaest}
    \sum_{p \le x_r} \bigg( \frac{b_{\ck,\ck'}(p)g_{\ck}(p)}{\sqrt{p}}
    +\frac{b_{\ck,\ck'}(p)^2g_{\ck}(p)}{2p}\bigg)=-\eta_{\ck,\ck'} \log \log x_r+O(1).
  \end{equation}
  Combining the preceding estimate along with \eqref{eq:eulerbd2} and \eqref{eq:markovsave} we have
  for each $r=1,\ldots, R-1$ and $b=b_{\ck}$ or $b=b_{\ck'}$
  \begin{equation} \label{eq:arbd}
    \mathcal A_r(b,\ck,\ck') \ll \bigg(\frac{N_{r+1}}{C_1^{1/3}} \bigg)^{N_{r+1}}  (\log x_r)^{-\eta_{\ck,\ck'}}.
  \end{equation}
  Applying \eqref{eq:arbd} in \eqref{eq:momentest1}, then using the resulting bound in \eqref{eq:mixedfirstbd} along with an analogous bound for the second term in \eqref{eq:mixedfirstbd} which follows by symmetry we get  \begin{equation} \label{eq:sumbd}
    \sum_{j \in S_r} \sqrt{\mathcal L(\ck,j) \mathcal L(\ck',j)} h(t_j) \ll T^2 \bigg(\frac{2e^2}{C_1^{1/3}} \bigg)^{N_{r+1}} (\log x_r)^{-\theta_{\ck,\ck'}}+T^{1+1/C_1+\varepsilon}
  \end{equation}
  where
  \[
    \theta_{\ck,\ck'}=\eta_{\ck,\ck'}+\delta_{\ck'}-\delta_{\ck}.
  \]

  Recalling the definition of $\eta_{\ck,\ck'}$ given in \eqref{eq:etadef}
  and that $\delta_{\ck}=1/2$ if $\ck^2=1$ and $\delta_{\ck}=1/4$ if $\ck^2 \neq 1$ , we get
  \begin{equation}
    \theta_{\ck,\ck'}= \begin{cases}
      1/4                       & \text{ if } \ck^2 \neq 1 \text{ and } (\ck')^2 \neq 1, \\
      3/8                       & \text{ if } \ck^2 \neq 1 \text{ and } (\ck')^2 = 1,    \\
      3/8                       & \text{ if } \ck^2 = 1 \text{ and } (\ck')^2 \neq 1,    \\
      1/2-\tfrac14 B_{\ck,\ck'} & \text{ if }  \ck^2=1 \text{ and } (\ck')^2=1.
    \end{cases}
  \end{equation}
  Since $d,d'$
  are squarefree and $d \neq d'$ we have  $B_{\ck,\ck'}$ is equal to either $0$ or $1$, so that
  \begin{equation} \label{eq:detabd}
    \theta_{\ck,\ck'} \ge \tfrac{1}{4}.
  \end{equation}
  Also note $N_{r+1}=2\lceil C_1 \log N_{r} \rceil$ and $C_1$ is sufficiently large so that
  \begin{equation} \label{eq:Nrbd}
    \left(\frac{2e^2}{C_1^{1/3}}\right)^{2 \lceil C_1 \log N_r \rceil}<N_r^{-2}.
  \end{equation}
  Using that $\log x_r=(\log T)/N_r^2$ and $\theta_{\ck,\ck'} \le 1/2$,  along with \eqref{eq:Nrbd} in \eqref{eq:sumbd}
  we get that
  \begin{equation} \label{eq:Rbdmid}
    \sum_{j \in S_r} \sqrt{\mathcal L(j,\ck) \mathcal L(j,\ck')} h(t_j) \ll\frac{T^2}{N_r (\log T)^{ \theta_{\ck,\ck'} }}.
  \end{equation}

  It remains to consider the contribution to the fractional moment for $j $ lying in $S_R$ or $S_0$. The contribution from these $j$ is bounded similar to before and we will only give brief arguments. Using Lemma \ref{lem:keyineq} and non-negativity as before we have
  \begin{equation} \label{eq:lastr}
    \sum_{j \in S_R} \sqrt{\mathcal L(j,\ck) \mathcal L(j,\ck')}  h(t_j)\ll
    \sum_{t_j>0} (\mathcal L(j,\ck) \mathcal M_R(\ck,\ck',j)+\mathcal L(j,\ck') \mathcal M_R(\ck',\ck,j)) h(t_j).
  \end{equation}
  Applying Lemma \ref{lem:expand} and using \eqref{eq:eulerbd2} along with \eqref{eq:etaest} we have
  \begin{equation} \label{eq:lastr2}
    \begin{split}
       & \sum_{t_j>0} \mathcal L(j,\ck) \mathcal M_R(\ck,\ck',j) h(t_j)                                                                                     \\
       & \ll T^2 (\log x_R)^{\delta_{\ck}-\delta_{\ck'}}
      \prod_{p \le x_R} \bigg(1+\frac{b_{\ck,\ck'}(p)g_{\ck}(p)}{\sqrt{p}}+\frac{b_{\ck,\ck'}(p)^2g_{\ck}(p^2)}{2p} +O\left(\frac{1}{p^{3/2}}\right) \bigg) \\
       & \ll T^2 (\log x_R)^{-\theta_{\ck,\ck'}} \ll T^2(\log T)^{-\theta_{\ck,\ck'}}.
    \end{split}
  \end{equation}
  where in the last step we used that $\log x_R=(\log T)/N_R^2 \ge (\log T)e^{-2C_1}$ by \eqref{eq:NRbd1}. Using \eqref{eq:lastr2} in \eqref{eq:lastr} and the analogous bound for the second term in \eqref{eq:lastr} which follows by symmetry we have
  \begin{equation} \label{eq:Rbdlast}
    \sum_{j \in S_R} \sqrt{\mathcal L(j,\ck) \mathcal L(j,\ck')} h(t_j) \ll T^2 (\log T)^{-\theta_{\ck,\ck'}}.
  \end{equation}

  Using Lemma \ref{lem:keyineq} and non-negativity once again we get
  \begin{equation} \label{eq:firstr2}
    \begin{split}
       & 2\sum_{j \in S_0} \sqrt{ \mathcal L(j,\ck) \mathcal L(j,\ck')}  h(t_j) \\
       & \qquad \le
      \sum_{t_j>0} \left(\mathcal L(j,\ck)+\mathcal L(j,\ck')\right)  \left( \left( \frac{P_{I_{1}}(b_{\ck},j) 2e^2}{N_{1}}\right)^{N_{1}}+\left(\frac{P_{I_{1}}(b_{\ck'},j) 2e^2}{N_{1}} \right)^{N_{1}} \right)
      h(t_j).
    \end{split}
  \end{equation}
  Applying Lemma \ref{lem:expand} and using \eqref{eq:markovsave} we get that
  \begin{equation} \label{eq:firstr}
    \begin{split}
      \sum_{t_j>0} \mathcal L(j,\ck) \bigg( \frac{P_{I_{1}}(b,j) 2e^2}{N_{1}}\bigg)^{N_{1}} h(t_j) \ll T^2 \bigg(\frac{N_1}{C_1^{1/3}} \bigg)^{N_1} \bigg(\frac{2 e^2}{N_1} \bigg)^{N_1}+T^{1+\varepsilon}  \ll T^2 e^{-N_1}
    \end{split}
  \end{equation}
  for each $b=b_{\ck},b_{\ck'}$.
  Using \eqref{eq:firstr} in \eqref{eq:firstr2} and recalling $C_1>100$ so $N_1 \ge 200 \log \log T$, we have
  \begin{equation} \label{eq:Rbdfirst}
    \sum_{j \in S_0} \sqrt{\mathcal L(j,\ck) \mathcal L(j,\ck')} h(t_j) \ll T^2 e^{-N_1} \le T^2(\log T)^{-200}.
  \end{equation}

  To complete the proof we use \eqref{eq:Rbdmid}, \eqref{eq:Rbdlast}, \eqref{eq:Rbdfirst} in \eqref{eq:fracbd} and then use \eqref{eq:recipbd} to bound the sum over $r$ to get that
  \[
    \sum_{T < t_{j} \le 2T} \sqrt{\mathcal L(j, \ck)\mathcal L(j, \ck')} \ll T^2 (\log T)^{-\theta_{\ck,\ck'}} \le T^2 (\log T)^{-1/4},
  \]
  where in the last step we used \eqref{eq:detabd}.
\end{proof}

\section{Proof of the main theorem} \label{finalproof}

In this section we prove the following Proposition, which together with Theorem \ref{thm:fracmoments} and \eqref{eq:phibd} will be used to prove Theorem \ref{thm:mainresult}.
\begin{proposition}\label{lem:local-bound-to-counting-bound} Let $\G=\slz$, and $z,w\in \H$. Assume that for $T \geq T_0$ we have
  \begin{equation} \label{eq5.1}
    \sum_{T < t_j\leq 2T}
    \abs{\phi_j(z)\phi_j(w)}=O
    \left(\frac{T^2}{(\log T)^{\eta}}\right)
  \end{equation}
  for some fixed $\eta >0$. Then
  \begin{equation}
    N(X;z,w)=\frac{2\pi}{\tmop{vol}(\GmodH)}X+O\left(\frac{X^{2/3}}{(\log X)^{2\eta/3}}\right).
  \end{equation}
\end{proposition}
\begin{remark} \label{remark5.2}
  We will also require a bound for the contribution of the continuous spectrum.  Known sup-norm bounds for the Eisenstein series yield sufficient bounds for our purposes. More precisely, Nordentoft \cite{Nordentoft:2021d}, improving on Young \cite{Young:2018}, proved the bound
  \begin{equation} \label{eq:supnormbd} E(z, 1/2 + it) \ll_{\varepsilon} y^{1/2}
    (|t| + 1)^{1/3+\varepsilon}.
  \end{equation}
\end{remark}
The proof of Proposition \ref{lem:local-bound-to-counting-bound} is quite standard, but we include a detailed proof for completeness.
Let $k(u)=\mathbf{1}_{u\leq (X-2)/4}$
be the characteristic kernel,
and  consider the point-pair invariant \begin{equation}
  K(z,w)=\sum_{\g\in \G}k(u(z,\g w)).
\end{equation}
We want to apply Selberg's pre-trace formula \cite[Theorem 7.4]{Iwaniec:2002} to $K(z,w)$, but as we shall see the Selberg--Harish-Chandra transform of $k$ (see \cite[Equation (1.62)]{Iwaniec:2002}) does not decay quickly enough for this to be applied directly. The Selberg--Harish-Chandra transform of the
indicator function $\mathbf{1}_{[0,(\cosh(R)-1)/2]}$
of the interval $[0,(\cosh(R)-1)/2]$  is given by
\begin{equation} \label{eq:hRdef}
  h_R(t)=2^{5/2}\int_0^R\sqrt{\cosh R-\cosh(r)}\cos(rt)dr=2\pi\sinh(R)P_{-\frac{1}{2}+it}^{-1}(\cosh R),
\end{equation}
where $P_\mu^\nu(z)$ is the associated Legendre function of the first kind (compare \cite[Equation (2.7)]{Chamizo:1996a}).
In particular, $h_R(t)$ is entire in $t$ and bounded by $O(R e^{R/2+R\abs{\Im(t)}})$.
Applying \cite[8.723.1]{GradshteynRyzhik:2007} we find that for $\Im(t)\leq 1-\e$ with $t\neq 0$ we have
\begin{equation}
  h_R(t)=\sqrt{2\pi\sinh R}(f(t)+f(-t)),
\end{equation}
with
\begin{equation}
  f(t)=\frac{\Gamma(it)}{\Gamma(3/2+it)}e^{itR}F\left(-\frac{1}{2},\frac{3}{2},1-it; \frac{1}{(1-e^{2R})}\right),
\end{equation}
where $F$ is Gauss' hypergeometric function.
For $R \geq \log \sqrt 2$, and $1\pm it$ bounded away from positive integers
the hypergeometric function is close to $1$ (see \cite[Equation (2.9)]{Chamizo:1996a}) and Stirling's formula gives the decay rate
\begin{equation} \label{eq:hRbd}
  h_R(t)=O\left(\frac{e^{\Re(it)R+R/2}}{(1+{\abs{t})}^{3/2}}\right).
\end{equation}

It turns out to be convenient smooth out the function $k$ to get better convergence on the spectral side. We do this as follows (compare \cite[Section 5]{PetridisRisager:2018}, \cite{Chamizo:1996a}):
Given kernels $k_1$ and $k_2$ we define their (hyperbolic) convolution as
\begin{equation}
  (k_1* k_2) (u(z,w))=\int_\H k_1(u(z,v))k_2(u(v,w))d\mu(v).
\end{equation}
The Selberg--Harish-Chandra transform
converts convolution into pointwise product
(see \cite[Equations (2.11)--(2.13)]{Chamizo:1996a}), i.e.
\begin{equation}
  \label{eq:convolution}
  h_{k_1*k_2}(t)=h_{k_1}(t) h_{k_2}(t).
\end{equation}
With this in mind, let $\delta>0$ be a small parameter sufficiently close to zero (to be chosen later), and consider
\begin{equation}
  k_\delta(u)= \frac{1}{4\pi \sinh^2(\delta/2)} \mathbf{1}_{[0,(\cosh \delta-1)/2]}(u).
\end{equation}
This kernel satisfies
\begin{equation}\int_\H k_\delta(u(z,w))d\mu(z)=1.
\end{equation}
Recall that $\cosh d_\H(z,w)=2u(z,w)+1$.
Now let $X>4$
and let $Y = \arccosh (X/2) >0$. We define two smooth approximations of the characteristic kernel
\begin{equation}
  k^{\pm}(u)=(\mathbf{1}_{[0,(\cosh(Y\pm \delta )-1)/2]}*k_\delta)(u).
\end{equation}
Using the triangle inequality for $d_\H$ we see (see also \cite[Equation (5.4)]{PetridisRisager:2017}) that
\begin{equation}
  k^-(u(z,w))\leq k(u(z,w))\leq k^+(u(z,w)).
\end{equation}
It follows that
\begin{equation}\label{eq:squeeze}
  K^-(X;z,w)\leq N(X;z,w)\leq K^+(X;z,w),
\end{equation}
where
\begin{equation} \label{defKplusminus}
  K^{\pm}(X;z,w):=\sum_{\g\in \G}k^{\pm}(u(z,\g w)).
\end{equation}
By \eqref{eq:convolution} we find that the Selberg--Harish-Chandra transforms of $k^{\pm}$ are given by
\begin{equation}
  h^{\pm}(t)=\frac{h_{Y\pm \delta}(t) h_\delta (t)}{4\pi\sinh^2(\delta/2)}.
\end{equation}
We have already analyzed $h_R$ for $R$ large. For $R=\delta$ small we analyze $h_\delta$ as in \cite[Lemma 2.4]{Chamizo:1996a}, \cite{PetridisRisager:2017} and using \eqref{eq:hRbd} we find that in the strip $\abs{\Im(t)}\leq 1-\e$,
\begin{equation} \label{eq:hpmdef}
  h^{\pm}(t)=O_{X,\delta}\left(\frac{1}{(1+\abs{t})^3}\right).
\end{equation}
This justifies that we can use the pre-trace formula. Finally, as in \cite[Equations (6.2)-(6.4)]{PetridisRisager:2018a} we find that
\begin{equation}\label{good-bounds-for-h}
  \begin{split}
    h^{\pm} \left(\tfrac{i}{2}\right) & = \pi X+O(1+\delta X),                                               \\
    h^{\pm}(t)                        & = O(X^{1/2}\min(\abs{t}^{-3/2}, \delta^{-3/2}\abs{t}^{-3}, \log X)).
  \end{split}
\end{equation}
We can now proceed to the proof of the Proposition.
\begin{proof}[Proof of Proposition \ref{lem:local-bound-to-counting-bound}]
  By Selberg's pre-trace formula \cite[Theorem~7.4]{Iwaniec:2002} and the definition of $K^{\pm}$ in \eqref{defKplusminus} we have
  \begin{equation}
    K^\pm(X;z,w)=2\sum_{t_j}h^{\pm}(t_j)\phi_j(z)\overline{\phi_j(w)}
    +
    \frac{2}{4\pi}\int_{\R} h^\pm(t)E \left(z,\tfrac{1}{2}+it \right)\overline{E \left(w,\tfrac{1}{2}+it \right)} dt.
  \end{equation}
  We note that the factor of 2 on the right-hand side comes from the fact that we count group elements rather than M\"obius transformations, i.e.~we count both $\pm \gamma$. We assume that  $\delta X\geq 1$.

  For $\slz$ the only small eigenvalue $\lambda\leq 1/4$ is the trivial eigenvalue $\lambda=0$ with constant eigenfunction $\phi_0(z)=\vol{\GmodH}^{-1/2}$. By \eqref{good-bounds-for-h} this gives a contribution to $K^\pm(X;z,w)$ of
  \begin{equation}
    \frac{2\pi X}{\vol{\GmodH}}+ O(\delta X).
  \end{equation}
  We split the rest of the spectrum in three parts: small i.e.  $\abs{t}, \abs{t_j}\leq T_0$; medium i.e. $T_0<\abs{t}, \abs{t_j}\leq \delta^{-1}$; and large i.e. $\delta^{-1}\leq \abs{t}, \abs{t_j}$. 

  For the small part of the spectrum we use $h^{\pm}(t)=O(X^{1/2}\log X)$ from
  \eqref{good-bounds-for-h} to find that the contribution to $K^\pm(X;z,w)$ is $\ll X^{1/2}\log X$. For the medium part we use the second bound in \eqref{good-bounds-for-h}, in particular
  \begin{equation}
    h^{\pm}(t)=O(X^{1/2}\abs{t}^{-3/2}).
  \end{equation}
  We find that the contribution from this part of the spectrum can be bounded by a constant times
  \begin{equation}
    X^{1/2}\left(\sum_{T_0<t_j\leq \frac{1}{\delta}}
    \frac{\abs{\phi_j(z)\phi_j(w)}}{t_j^{3/2}}
    +
    \int_{T_0}^{ \frac{1}{\delta}}
    \frac{|E(z,\tfrac{1}{2}+it)E(w,\tfrac{1}{2}+it)|}{t^{3/2}}
    dt\right).
  \end{equation}
  Using a dyadic decomposition, the contribution from the discrete spectrum is bounded by a constant times
  \begin{eqnarray}
    X^{1/2}\sum_{\log_2(T_0)< n\leq \log_2(1/\delta)}\frac{1}{2^{3n/2}}\sum_{2^n<t_j\leq  2^{n+1}}
    |\phi_j(z)\phi_j(w)|.
  \end{eqnarray}
  Using our assumption, this is bounded by
  \begin{eqnarray}
    X^{1/2} \sum_{\log_2(T_0)<n\leq \log_2(1/\delta)}\frac{1}{2^{3n/2}}\frac{2^{2n}}{n^\eta}\ll X^{1/2}\frac{\delta^{-1/2}}{(\log(1/\delta) )^{\eta}}.
  \end{eqnarray}
  Following Remark \ref{remark5.2} and using \eqref{eq:supnormbd}, the contribution from the continuous spectrum is bounded by $X^{1/2} \delta^{-1/6 +\varepsilon}$.

  For the large part of the spectrum we use the second estimate from \eqref{good-bounds-for-h}, i.e. $h^{\pm}(t) = O (X^{1/2} \delta^{-3/2} \abs{t}^{-3})$, to find that the contribution from the discrete spectrum is bounded by a constant times
  \begin{eqnarray}
    && X^{1/2} \sum_{\frac{1}{\delta}<t_j}\frac{1}{\delta^{3/2}t_j^{3}}\abs{\phi_j(z)\phi_j(w)}
    \ll X^{1/2} \sum_{n=0}^{\infty}\frac{\delta^{3/2} }{2^{3n}}\sum_{2^n\frac{1}{\delta}<t_j\leq 2^{n+1}\frac{1}{\delta}}
    \abs{\phi_j(z)\phi_j(w)},
  \end{eqnarray}
  which is bounded by a constant times
  \begin{eqnarray}
    X^{1/2}\delta^{3/2}\sum_{n=0}^{\infty}\frac{1}{2^{3n}}\frac{2^{2n}\delta^{-2}}{(\log(1/\delta) )^{\eta}}\ll X^{1/2}\frac{\delta^{-1/2}}{(\log(1/\delta) )^{\eta}}.
  \end{eqnarray}
 The contribution from the continuous spectrum is again bounded by $X^{1/2} \delta^{-1/6 +\varepsilon}$. Adding all the contributions we find that
  \begin{equation}
    K^\pm(X;z,w)=\frac{2\pi X}{\vol{\GmodH}}
    +
    O\left(X\delta +
    \frac{X^{1/2}}{\delta^{1/2}(\log(1/\delta) )^{\eta}}
    + X^{1/2}\log X+X^{1/2} \delta^{-1/6 +\varepsilon}\right),
  \end{equation} and by using \eqref{eq:squeeze} we get
  \begin{equation}
    N(X;z,w)=\frac{2\pi X}{\vol{\GmodH}}
    +
    O\left(X\delta +
    \frac{X^{1/2}}{\delta^{1/2}(\log(1/\delta) )^{\eta}}
    + X^{1/2}\log X+X^{1/2} \delta^{-1/6 +\varepsilon}\right).
  \end{equation}
  Choosing $\delta^{-1}=X^{1/3}(\log X)^{2\eta/3}$ gives the result.
\end{proof}

\begin{proof}[Proof of Theorem \ref{thm:mainresult}]
  Given negative, squarefree integers $d,d'$ with $d,d' \equiv 1 \pmod 4$,
  let $K=\mathbb Q(\sqrt{d})$ and $K'=\mathbb Q(\sqrt{d'})$.
  Applying \eqref{eq:phibd} and then using Theorem \ref{thm:fracmoments} we have for $T \ge 1$
  \begin{equation}\label{0605:eq002}
    \begin{split}
      \sum_{T < t_j \le 2T} \!\!\! |\phi_j(z_{d})\phi_j(z_{d'})|
       & \ll \!\!\!\!
      \sum_{\substack{\ck \in \widehat{\Cl}_{K} \\\;\ck' \in  \widehat{\Cl}_{K'}}} \!\!\!
      \sum_{\;T < t_j \le 2T} \!\!\!\!\!
      \frac{\sqrt{L(\tfrac12, \phi_j \times f_{\ck})L(\tfrac12,\phi_j \times f_{\ck'})}}{L(1,\tmop{sym}^2 \phi_j)}
      \ll
      \frac{T^2}{(\log T)^{1/4}}.
    \end{split}
  \end{equation}
  This shows that the assumption of Proposition \ref{lem:local-bound-to-counting-bound} is satisfied for $\eta=1/4$
  and the result immediately follows.
\end{proof}

\section{The second moment of the error term} \label{secondmomenterror}

Let us now move to discuss the proof of Corollary \ref{corollary1.5}.
As in the proof of the pointwise bound in the previous section, our starting point is an approximation of the counting function $N(X;z,w)$
from above and below by two point-pair invariants $K^\pm(X;z,w)$, see \eqref{eq:squeeze},
constructed to have better regularity and decay properties.
These functions are defined in terms of a small parameter $0<\delta\ll 1$
which will be optimized later. We anticipate that we will take $\delta\geq X^{-2/3}$.

Let $f$ be a smooth function compactly supported on $[1/2,5/2]$ with
$f(y)\geq 1$ when $y\in [1,2]$. Let $X\gg 1$.
Subtracting the main term in \eqref{eq:squeeze} and integrating over $x\in[X,2X]$ shows that
\begin{equation}\label{0605:eq003}
  \frac{1}{X}\int_{X}^{2X} \biggl|N(x;z,w)-\frac{2\pi x}{\vol{\GmodH}}\biggr|^{2} \! dx
  \ll
  \max_{\pm}\frac{1}{X} \int_\RR \biggl|K^\pm(x;z,w)-\frac{2\pi x}{\vol{\GmodH}}\biggr|^{2} \! f\left(\frac{x}{X}\right) \! dx.
\end{equation}
We will show that the right-hand side above is $O(X(\log X)^{3/4})$
when $z = z_d, w=z_{d'}$ are Heegner points of distinct discriminants as in the statement of the corollary.

We use the pre-trace formula to  write an absolutely convergent expansion for $K^\pm(x;z,w)$ as defined in \eqref{defKplusminus}
in terms of cusp forms and Eisenstein series. Recall that by \eqref{good-bounds-for-h} the contribution of the 
eigenvalue $\lambda=0$ 
is $ \frac{2\pi x}{\vol{\GmodH}}+ O(\delta x)$ for $1/2 \le x/X \le 5/2$ since $\delta \ge X^{-2/3}$.
Squaring and separating the discrete part and the continuous part, we deduce that
\begin{equation}\label{0605:eq004}
  \begin{split}
    \frac{1}{X} & \int_{\mathbb{R}} \left| K^{\pm} (x;z, w)  - \frac{2 \pi x}{\vol{\GmodH}}\right|^2 f\left(\frac{x}{X}\right)dx
    \\
                & \quad\ll
    \frac{1}{X} \int_{\mathbb{R}}
    \left|
    \sum_{t_j \in \mathbb{R}}h^{\pm}(t_j)\phi_j(z)\overline{\phi_j(w)}
    \right|^2 f\left(\frac{x}{X}\right)dx
    \\
                & \phantom{xxxxxxx}+
    \frac{1}{X} \int_{\mathbb{R}}
    \left|\int_{\mathbb{R}} h^\pm(t)E \left(z,\tfrac{1}{2}+it \right)\overline{E \left(w,\tfrac{1}{2}+it \right)} dt \right|^2
    f\left(\frac{x}{X}\right)
    dx + \delta^2 X^2,
  \end{split}
\end{equation}
where $h^{\pm}$ is as defined in \eqref{eq:hpmdef}.
Before performing the integration in $x$, we are going to simplify the function $h^\pm(t)$.
Notice that $h^{\pm}$ is even, so replacing $t\mapsto -t$ in the $t$-integral amounts to taking the complex conjugate
and therefore it suffices to estimate the contribution of $t>0$.
Recall from Section \ref{finalproof}
that we write $Y=\arccosh(x/2)$ and let $h_R(t)$ be as in \eqref{eq:hRdef}.
Rather than bounding in absolute value,
following \cite[Lemma 2.4 (a)]{Chamizo:1996a} we write, for $t>0$,
\begin{equation}\label{eq:h_Y-exp}
  h_{Y\pm\delta}(t) = \frac{2\sqrt{2\pi\sinh (Y\pm \delta)}}{t^{3/2}} \cos(t(Y\pm \delta)-\tfrac{3\pi}{4}) + O(t^{-5/2}x^{1/2}).
\end{equation}
By the Taylor expansions of cosine, $\sinh$ and $\arccosh$ we have
\[
  \begin{gathered}
    \sqrt{2\sinh (Y\pm \delta)} = x^{1/2} e^{\pm\delta/2} + O(x^{-3/2})\\
    \cos(t(Y\pm\delta)-\tfrac{3\pi}{4}) = \cos(t(\pm\delta + \log x) - \tfrac{3\pi}{4}) + O(\min(1,tx^{-2})).
  \end{gathered}
\]
For later use, we bound
$\min(1,tx^{-2})\leq t^{1/2}x^{-1}$.
If we set
\begin{equation}\label{def:HX}
  H_x(t) = \frac{2\sqrt{\pi}e^{\pm\delta/2}}{t^{3/2}} \Re(x^{\frac 12+it}e^{\pm\delta ti-\frac{3\pi i}{4}}),
\end{equation}
we obtain
\[
  h_{Y\pm\delta}(t) = H_x(t) + O(t^{-3/2}(1+t^{1/2})x^{-1/2} + t^{-5/2}x^{1/2}),
\]
where the last error comes from \eqref{eq:h_Y-exp}.
Multiplying by $h_\delta(t)(4\pi \sinh^2(\delta/2))^{-1}$
and noting that
\begin{equation}\label{eq:hdeltabd}
h_\delta(t)(4\pi \sinh^2(\delta/2))^{-1}\ll\min(1,(\delta t)^{-3/2})\ll(1+\delta t)^{-3/2},
\end{equation} (compare \cite[Equations (6.2)-(6.4)]{PetridisRisager:2018a}, see also \cite[Equation (4.20)]{Cherubini:2018})
we arrive at the expression
\[
  h^\pm(t) = \frac{H_x(t)h_\delta(t)}{4\pi \sinh^2(\delta/2)} + R_x(t),\quad R_x(t)=O\left(\frac{(1+t^{1/2})x^{-1/2}}{t^{3/2}(1+\delta t)^{3/2}}+\frac{x^{1/2}}{t^{5/2}}\right).
\]

Summing the error term above against the cusp forms or integrating against the Eisenstein series
gives a negligible contribution. Indeed, we have
\[
  \sum_{t_j>0} R_x(t) \abs{\phi_j(z)\phi_j(w)}
  \ll
  (x^{-1/2}\delta^{-3/2} + x^{1/2})
  \sum_{t_j>0} \frac{|\phi_j(z)\phi_j(w)|}{t_j^{\frac 52}}
  \ll
  x^{1/2},
\]
since the sum converges and $\delta$ will be chosen such that $\delta \gg x^{-2/3}$.
Similarly for the continuous spectrum, we can bound the product of the two Eisenstein series 
by $O_{z,w}((1+t)^{2/3+\varepsilon})$ by \eqref{eq:supnormbd}.
Notice that near zero we also have $E(z,\frac 12+it)=O_z(t)$ for all points $z$ since $E(z,\frac 12)$
vanishes identically
and the function is regular. Therefore, we have
\[
  \int_{0}^\infty R_x(t) E(z,\tfrac 12+it)\overline{E(w,\tfrac 12+it)} dt
  \ll
  (x^{-1/2}\delta^{-3/2} + x^{1/2})
  \int_{0}^\infty \frac{1}{t^{\frac 12}(1+t)^{\frac{4}{3}-\varepsilon}} dt \ll x^{1/2},
\]
since again the integral converges and $\delta \gg x^{-2/3}$.
Summarizing, so far we showed that
\begin{equation}\label{0605:eq001}
  \begin{split}
    \frac{1}{X} \int_\RR & \biggl|K^\pm(x;z,w)-\frac{2\pi x}{\vol{\GmodH}}\biggr|^{2} f\left(\frac{x}{X}\right) dx
    \\
                         & \ll
    \delta^2X^2+X+
    \frac{1}{X} \int_\RR \biggl|\sum_{t_j>0} \frac{H_x(t_j)h_\delta(t_j)}{4\pi \sinh^2(\delta/2)} \phi_j(z)\overline{\phi_j(w)}\biggr|^{2} f\left(\frac{x}{X}\right) dx
    \\
                         & \phantom{xxxxxxxxx}+
    \frac{1}{X} \int_\RR \biggl|\int_0^\infty \frac{H_x(t)h_\delta(t)}{4\pi \sinh^2(\delta/2)} E(z,\tfrac{1}{2}+it)\overline{E(w,\tfrac{1}{2}+it)} \,dt\biggr|^{2} f\left(\frac{x}{X}\right) dx.
  \end{split}
\end{equation}
We open the squares and focus on the $x$-integral.
This involves the function $f(x/X)$ and two copies of $H_x(t)$.

\begin{lemma}\label{0405:lemma1}
  Let $X\gg 1$, $H_x(t)$ be as in \eqref{def:HX}, $f$ be as above and let also $t_1,t_2>0$. Then for every $N\geq 0$ we have
  \[
    \frac{1}{X} \int_\mathbb{R} H_x(t_1)\overline{H_x(t_2)} f\left(\frac{x}{X}\right) dx
    \ll_{f,N}
    \frac{X}{(t_1t_2)^{3/2}(1+|t_1-t_2|)^N}.
  \]
\end{lemma}

\begin{proof}
  This follows by integrating by parts $N$ times the term $x^{1\pm i(t_1\pm t_2)}$
  coming from the definition of $H_x$ and differentiating the function $f$.
\end{proof}

With Lemma \ref{0405:lemma1} at hand we can bound
the discrete and the continuous contribution appearing in \eqref{0605:eq001}.
We do this separately, starting with the Eisenstein series.

\begin{lemma}\label{0605:lemma2}
  Let $X \gg 1$, $f$ be as above, $0<\delta\ll 1$, $H_x$ be as in \eqref{def:HX} and let $z,w$ be any two points in the hyperbolic plane. We have
  \[
    \frac{1}{X} \int_\RR \biggl|\int_0^\infty \frac{H_x(t)h_\delta(t)}{4\pi \sinh^2(\delta/2)} E(z,\tfrac{1}{2}+it)\overline{E(w,\tfrac{1}{2}+it)} \,dt\biggr|^{2} f\left(\frac{x}{X}\right) dx
    \ll_{z,w} X.
  \]
\end{lemma}

\begin{proof}
  Let $J$ denote the integral in the statement.
  Open the square and interchange the integrals (everything converges absolutely).
  Applying Lemma \ref{0405:lemma1} to the $x$-integral with $N=1$,
  bounding $h_\delta(t_1)h_\delta(t_2)(4\pi \sinh^2(\delta/2))^{-2}$
  by $O(1)$ using \eqref{eq:hdeltabd},
  and recalling the bound for Eisenstein series given by \eqref{eq:supnormbd}
  as well as the observation we made earlier about the behaviour near zero,
  we obtain
  \[
    J\ll_{z,w,\varepsilon} X \int_{0}^\infty\int_{0}^\infty \frac{dt_1dt_2 }{(1+t_1)^{\frac 56-\varepsilon}(1+t_2)^{\frac 56-\varepsilon}(1+|t_1-t_2|)} \ll_{z,w} X.
    \qedhere
  \]
\end{proof}

Finally, we look at the contribution of the discrete spectrum
and we restrict to Heegner points.

\begin{lemma}\label{0605:lemma3}
  Let $X \gg 1$, $f$ be as above, $0<\delta\ll 1$, $H_x$ be as in \eqref{def:HX} and let $z_d,z_{d'}$ be Heegner points 
  of different squarefree discriminants $d,d'<0$ respectively.
  We have
  \[
    \frac{1}{X} \int_\RR \biggl|\sum_{t_j>0} \frac{H_x(t_j)h_\delta(t_j)}{4\pi \sinh^2(\delta/2)} \phi_j(z_d)\overline{\phi_j(z_{d'})}\biggr|^{2} f\left(\frac{x}{X}\right) dx
    \ll_{z_d,z_{d'}}
    X  \left(\log \frac{1}{\delta}\right)^{3/4}.
  \]
\end{lemma}
\begin{proof}
  Let $S$ denote the left-hand side in the statement.
  Opening the square and applying Lemma \ref{0405:lemma1} with $N\geq 2$
  and the estimate \eqref{eq:hdeltabd} we get
  \[
    S \ll \, X \! \sum_{t_j,t_l>0} \frac{a_ja_l}{(1+|t_j-t_l|)^N},
  \]
  where
  \[
    a_j = \frac{|\phi_j(z_d)\phi_j(z_{d'})|}{(t_j(1+\delta t_j))^{3/2}}
  \]
  and similarly for $a_l$.
  By symmetry, it suffices to estimate the sum with $t_l\geq t_j$.
  For a given $t_j$, we can bound
  \begin{equation} \label{boundal}
    \sum_{t_l\geq t_j} \frac{a_l}{(1+|t_l-t_j|)^N}
    \le
    \sum_{n=0}^\infty \frac{1}{(n+1)^N} \sum_{n\leq t_l-t_j\leq n+1} a_l.
  \end{equation}
  By the local Weyl law in unit intervals and the crude estimates $t_l\geq t_j+n\geq t_j$
  and $(1+\delta t_l)\geq 1$, we can bound \eqref{boundal} by 
  $O_{N,z_d,z_{d'}}(t_j^{-1/2})$. This leads to
  \[
    \frac{S}{X} \ll_{N,z_d,z_{d'}} \sum_{t_j>0} \frac{a_j}{t_j^{1/2}} = \sum_{t_j>0} \frac{|\phi_j(z_d)\phi_j(z_{d'})|}{t_j^2(1+\delta t_j)^{3/2}}.
  \]
  When $t_j\geq \delta^{-1}$, by doing a dyadic decomposition
  and applying the Weyl law we can bound
  \[
  \begin{split}
    \sum_{t_j\geq \delta^{-1}} \frac{|\phi_j(z_d)\phi_j(z_{d'})|}{t_j^2(1+\delta t_j)^{3/2}}
   & \ll
    \frac{1}{\delta^{3/2}} \sum_{n=0}^\infty \; \sum_{2^n\leq \delta t_j\leq 2^{n+1}} \frac{|\phi_j(z_{d})\phi_j(z_{d'})|}{t_j^{7/2}} \\
   & \ll_{z_d,z_{d'}}
    \frac{1}{\delta^{3/2}}\sum_{n=0}^\infty \left(\frac{\delta}{2^{n}}\right)^{3/2} \ll_{z_d,z_{d'}} 1.
    \end{split}
  \]
  When $t_j\leq \delta^{-1}$ we use a dyadic decomposition then apply \eqref{0605:eq002}
  and obtain
  \[
    \sum_{t_j\leq \delta^{-1}} \frac{|\phi_j(z_d)\phi_j(z_{d'})|}{t_j^2} \ll_{z_d,z_{d'}} \left(\log\frac{1}{\delta}\right)^{3/4}.
  \]
  Combining the two parts of the series we obtain $S=O_{z_d,z_{d'}}(X (\log\tfrac{1}{\delta})^{3/4})$, as claimed.
\end{proof}

\begin{proof}[Proof of Corollary \ref{corollary1.5}]
  Combining \eqref{0605:eq003}, \eqref{0605:eq004} and \eqref{0605:eq001}
  with Lemma \ref{0605:lemma2} and Lemma \ref{0605:lemma3} we obtain
  \[
    \frac{1}{X}\int_{X}^{2X} \biggl|N(x;z_d, z_{d'})-\frac{2\pi x}{\vol{\GmodH}}\biggr|^{2} dx
    \ll X + \delta^2 X^2 + X(\log\tfrac{1}{\delta})^{3/4}.
  \]
  We take $\delta=X^{-1/2}$ and obtain the claimed bound.
\end{proof}

\bibliographystyle{amsplain}

\end{document}